\documentclass[12pt]{amsart}

\usepackage[matrix,arrow,curve,frame]{xy}
\usepackage{amsmath,amsthm,amssymb,enumerate}
\usepackage{latexsym}
\usepackage{amscd}
\usepackage[colorlinks=false]{hyperref}
\usepackage{euscript}

\newtheorem{thm}[subsection]{Theorem}
\newtheorem{defn}[subsection]{Definition}

\newtheorem{cor}[subsection]{Corollary}
\newtheorem{lemma}[subsection]{Lemma}
\newtheorem{conj}[subsection]{Conjecture}
\newtheorem{remark}[subsection]{Remark}

\theoremstyle{definition}

\def\cP{{\cal P}}
\def\cA{{\cal A}}

\def\ra{\rightarrow}
\def\bra{\langle}
\def\ket{\rangle}

\def\cA{{\mathcal A}}
\def\cB{{\mathcal B}}

\def\cE{{\mathcal E}}

\def\cH{{\mathcal H}}
\def\cI{{\mathcal I}}
\def\cJ{{\mathcal J}}

\def\cL{{\mathcal L}}
\def\cM{{\mathcal M}}

\def\cP{{\mathcal P}}

\def\cR{{\mathcal R}}
\def\cS{{\mathcal S}}

\def\cV{{\mathcal V}}
\def\cW{{\mathcal W}}

\def\gg{{\mathfrak g}}
\def\gh{{\mathfrak h}}

\def\gl{{\mathfrak l}}

\def\gs{{\mathfrak s}}

\newfont{\german}{eufm10}

\begin{document}
\pagestyle{plain}

\title
{Invariant theory and the Heisenberg vertex algebra}

\author{Andrew R. Linshaw}
\address{Fachbereich Mathematik, Technische Universit\"at Darmstadt, 64289 Darmstadt, Germany.}
\email{linshaw@mathematik.tu-darmstadt.de}

{\abstract
\noindent The invariant subalgebra $\cH^+$ of the Heisenberg vertex algebra $\cH$ under its automorphism group $\mathbb{Z}/2\mathbb{Z}$ was shown by Dong-Nagatomo to be a $\cW$-algebra of type $\cW(2,4)$. Similarly, the rank $n$ Heisenberg vertex algebra $\cH(n)$ has the orthogonal group $O(n)$ as its automorphism group, and we conjecture that $\cH(n)^{O(n)}$ is a $\cW$-algebra of type $\cW(2,4,6,\dots,n^2+3n)$. We prove our conjecture for $n=2$ and $n=3$, and we show that this conjecture implies that $\cH(n)^G$ is strongly finitely generated for any reductive group $G\subset O(n)$. }

\keywords{invariant theory; vertex algebra; reductive group action; orbifold construction; strong finite generation; $\cW$-algebra}
\maketitle
\tableofcontents
\section{Introduction}

We call a vertex algebra $\cV$ {\it strongly finitely generated} if there exists a finite set of generators such that the collection of iterated Wick products of the generators and their derivatives spans $\cV$. This property has several important consequences, and in particular implies that both Zhu's associative algebra $A(\cV)$, and Zhu's commutative algebra $\cV / C_2(\cV)$, are finitely generated. By an {\it invariant vertex algebra}, we mean a subalgebra $\cV^G\subset \cV$, where $G$ is a group of automorphisms of $\cV$. It is our belief that if $\cV$ is a simple, strongly finitely generated vertex algebra, and $G$ is reductive, $\cV^G$ will be strongly finitely generated under fairly general circumstances. Isolated examples of this phenomenon have been known for some years (see for example \cite{BFH}\cite{FKRW}\cite{EFH}\cite{DNI}\cite{KWY}), although the first general results of this kind were obtained in \cite{LII}, in the case where $\cV$ is the $\beta\gamma$-system $\cS(V)$, $bc$-system $\cE(V)$, or $bc\beta\gamma$-system $\cE(V)\otimes \cS(V)$, associated to $V = \mathbb{C}^n$. The strong finite generation property is a subtle and essentially \lq\lq quantum" phenomenon, and is generally destroyed by passing to the classical limit before taking invariants. Often, $\cV$ admits a $G$-invariant filtration for which $gr(\cV)$ is a commutative algebra with a derivation (i.e., an abelian vertex algebra), and the classical limit $gr(\cV^G)$ is isomorphic to $(gr(\cV))^G$ as a commutative algebra. Unlike $\cV^G$, $gr(\cV^G)$ is generally not finitely generated as a vertex algebra, and a presentation will require both infinitely many generators and infinitely many relations.

One of the most basic examples of an invariant vertex algebra was studied by Dong-Nagatomo in \cite{DNI}. Let $\cH$ denote the Heisenberg vertex algebra, which is generated by a field $\alpha$ satisfying the operator product expansion (OPE) relation $\alpha(z) \alpha(w)\sim (z-w)^{-2}$. Clearly the automorphism group $Aut(\cH)$ is isomorphic to $\mathbb{Z}/2\mathbb{Z}$, and is generated by the involution $\theta$ sending $\alpha\mapsto -\alpha$. In \cite{DNI} it was shown that the invariant subalgebra $\cH^+$ under $\theta$ is a $\cW$-algebra of type $\cW(2,4)$, and in particular is strongly generated by the Virasoro element $L$ and an element $J$ of weight four. Using this result, the authors described the Zhu algebra of $\cH^+$, which is a commutative algebra on two generators, and they classified the irreducible modules of $\cH^+$. In \cite{DNII}, Dong-Nagatomo considered a higher-rank analogue $\cH(n)^+$ of $\cH^+$. In this notation, $\cH(n)$ is the rank $n$ Heisenberg algebra, which is just the tensor product of $n$ copies of $\cH$, and $\cH(n)^+$ is the invariant subalgebra under the $-1$ involution. Unlike the rank $1$ case, the Zhu algebra of $\cH(n)^+$ is nonabelian, and it is difficult to describe it completely by generators and relations. However, the authors obtained enough information about it to classify the irreducible modules of $\cH(n)^+$. This result is important for understanding the structure and representation theory of vertex algebras of the form $V_L^+$. Here $V_L$ is the lattice vertex algebra attached to some lattice $L$ of rank $n$, and $V_L^+$ is the invariant subalgebra under the $-1$ involution.

In this paper, we study general invariant vertex algebras $\cH(n)^G$, where $G$ is an arbitrary reductive group of automorphisms of $\cH(n)$. In the case where the action of $G$ extends to $V_L$, an understanding of $\cH(n)^G$ is a necessary first step in studying $V_L^G$. The {\it full} automorphism group of $\cH(n)$ preserving a natural conformal structure is the orthogonal group $O(n)$. We begin by studying $\cH(n)^{O(n)}$, which coincides with $\cH^+$ in the case $n=1$. For an arbitrary reductive group $G\subset O(n)$, $\cH(n)^G$ is completely reducible as a module over $\cH(n)^{O(n)}$, and this module structure is an essential ingredient in our description. Our approach in this paper is quite parallel to our earlier study of invariant subalgebras of the $\beta\gamma$-system $\cS(V)$. The automorphism group of $\cS(V)$ preserving its conformal structure is $GL_n$, and $\cS(V)^{GL_n}$ is isomorphic to $\cW_{1+\infty,-n}$ by a theorem of Kac-Radul \cite{KR}. In \cite{LI}, we studied $\cW_{1+\infty,-n}$ via classical invariant theory, and we use a similar method in Section \ref{secortho} of this paper to study $\cH(n)^{O(n)}$. There are many parallels between these two vertex algebras; for example, they both have abelian Zhu algebras, which implies that their irreducible, admissible modules are all highest-weight modules. This observation is crucial in our description in Section \ref{secgeneral} of $\cH(n)^G$ as a module over $\cH(n)^{O(n)}$, which uses essentially the same ideas as \cite{LII}. 

First of all, there is an $O(n)$-invariant filtration on $\cH(n)$ such that the associated graded object $gr(\cH(n))$ is isomorphic to $S=Sym \bigoplus_{j\geq 0} V_j$ as a commutative ring, where $V_j\cong \mathbb{C}^n$ as $O(n)$-modules. In fact, $\cH(n)\cong S$ as vector spaces, and we view $\cH(n)$, equipped with the Wick product, as a deformation of $S$. Using Weyl's first and second fundamental theorems of invariant theory for the standard representation of $O(n)$, we obtain a natural (infinite) strong generating set $\{\omega_{a,b}|~0\leq a\leq b\}$ for $\cH(n)^{O(n)}$, as well as an infinite set of relations among these generators. A linear change of variables produces a slightly more economical set of strong generators $\{j^{2m}|~m\geq 0\}$, where $j^{2m}$ has weight $2m+2$. In fact, $\cH(n)^{O(n)}$ is generated as a vertex algebra by $\{j^0, j^2\}$ for all $n\geq 1$, although this is only a strong generating set in the case $n=1$. The relation of minimal weight among $\{j^{2m}|~m\geq 0\}$ and their derivatives occurs at weight $n^2+3n+2$, and we conjecture that it gives rise to a decoupling relation $$j^{n^2+3n} = P(j^0,j^2,\dots,j^{n^2+3n-2}).$$ Here $P$ is a normally ordered polynomial in $j^0, j^2,\dots,j^{n^2+3n-2}$, and their derivatives. An easy consequence of our conjecture is that higher decoupling relations of the form $j^{2r} = Q_{2r}(j^0,j^2,\dots,j^{n^2+3n-2})$ exist for all $r\geq\frac{1}{2} (n^2+3n)$. Hence $\cH(n)^{O(n)}$ has a minimal strong generating set $\{j^0, j^2, \dots, j^{n^2+3n-2}\}$, and in particular is a $\cW$-algebra of type $\cW(2,4,6,\dots, n^2+3n)$. By computer calculation, we prove our conjecture for $n=2$ and $n=3$, but we are unable to prove it in general.

 By a fundamental result of Dong-Li-Mason \cite{DLM}, $\cH(n)$ has a decomposition of the form $$ \cH(n) \cong \bigoplus_{\nu\in H} L(\nu)\otimes M^{\nu},$$ where $H$ indexes the irreducible, finite-dimensional representations $L(\nu)$ of $O(n)$, and the $M^{\nu}$'s are inequivalent, irreducible, admissible $\cH(n)^{O(n)}$-modules. Since the Zhu algebra of $\cH(n)^{O(n)}$ is abelian, each $M^{\nu}$ above is a highest-weight module. For any reductive group $G\subset O(n)$, $\cH(n)^G$ is also a direct sum of irreducible, highest-weight $\cH(n)^{O(n)}$-modules. Using a classical theorem of Weyl, we show that there is a finite set of irreducible $\cH(n)^{O(n)}$-submodules of $\cH(n)^G$ whose direct sum contains an (infinite) strong generating for $\cH(n)^G$. Since $\cH(n)^{O(n)}$ is finitely generated, this shows that $\cH(n)^G$ is finitely generated as a vertex algebra. Finally, assuming our conjecture that $\cH(n)^{O(n)}$ is strongly finitely generated, we show that $\cH(n)^G$ is strongly finitely generated as well. Since our conjecture holds for $n=2$ and $n=3$, the strong finite generation of $\cH(2)^G$ and $\cH(3)^G$ for an arbitrary reductive $G$ is an immediate consequence.

There is an application of these results to invariant subalgebras of affine vertex algebras which we develop in a separate paper \cite{LIII}. Let $\gg$ be a simple, finite-dimensional Lie algebra, and let $V_k(\gg)$ denote the universal affine vertex algebra at level $k$ associated to $\gg$. It is freely generated by vertex operators $X^{\xi}$, which are linear in $\xi\in\gg$ and satisfy the OPE relations $$X^{\xi}(z) X^{\eta}(w) \sim k\bra \xi,\eta\ket (z-w)^{-2} + X^{[\xi,\eta]}(w)(z-w)^{-1},$$ where $\bra ,\ket$ denotes the normalized Killing form $\frac{1}{2 h^{\vee}} \bra ,\ket_K$. Let $G$ be a reductive group of automorphisms of $V_k(\gg)$ for all $k\in \mathbb{C}$. In particular, $G$ acts on the weight-one subspace $V_k(\gg)[1]\cong \gg$, and $G$ preserves both the bracket and the bilinear form on $\gg$. Therefore $G$ lies in the orthogonal group $O(n)$ for $n= \text{dim}(\gg)$, so $G$ also acts on the Heisenberg algebra $\cH(n)$. As vector spaces, we have $V_k(\gg)^G \cong (Sym \oplus_{j\geq 0} V_j)^G\cong \cH(n)^G$, where $V_j\cong \mathbb{C}^n$ for all $j\geq 0$, and we regard $\cH(n)^G$ as a \lq\lq partial abelianization" of $V_k(\gg)^G$. Invariant subalgebras of $V_k(\gg)$ are much more complicated and difficult to study than invariant subalgebras of $\cH(n)$, but for generic values of $k$, a strong generating set for $\cH(n)^G$ gives rise to a strong generating set for $V_k(\gg)^G$. Therefore the conjectured strong finite generation of $\cH(n)^{O(n)}$ has a far-reaching consequence; it implies that $V_k(\gg)^G$ is strongly finitely generated for generic values of $k$. Finally, since our conjecture holds for $n=3$, this statement is true in the case $\gg = \gs\gl_2$.

\section{Vertex algebras}
In this section, we define vertex algebras, which have been discussed from various different points of view in the literature \cite{B}\cite{FBZ}\cite{FHL}\cite{FLM}\cite{K}\cite{LiI}\cite{LZ}. We will follow the formalism developed in \cite{LZ} and partly in \cite{LiI}. Let $V=V_0\oplus V_1$ be a super vector space over $\mathbb{C}$, and let $z,w$ be formal variables. By $QO(V)$, we mean the space of all linear maps $$V\ra V((z)):=\{\sum_{n\in\mathbb{Z}} v(n) z^{-n-1}|
v(n)\in V,\ v(n)=0\ \text{for} \ n>>0 \}.$$ Each element $a\in QO(V)$ can be
uniquely represented as a power series
$$a=a(z):=\sum_{n\in\mathbb{Z}}a(n)z^{-n-1}\in End(V)[[z,z^{-1}]].$$ We
refer to $a(n)$ as the $n$th Fourier mode of $a(z)$. Each $a\in
QO(V)$ is of the shape $a=a_0+a_1$ where $a_i:V_j\ra V_{i+j}((z))$ for $i,j\in\mathbb{Z}/2\mathbb{Z}$, and we write $|a_i| = i$.

On $QO(V)$ there is a set of nonassociative bilinear operations
$\circ_n$, indexed by $n\in\mathbb{Z}$, which we call the $n$th circle
products. For homogeneous $a,b\in QO(V)$, they are defined by
$$
a(w)\circ_n b(w)=Res_z a(z)b(w)~\iota_{|z|>|w|}(z-w)^n-
(-1)^{|a||b|}Res_z b(w)a(z)~\iota_{|w|>|z|}(z-w)^n.
$$
Here $\iota_{|z|>|w|}f(z,w)\in\mathbb{C}[[z,z^{-1},w,w^{-1}]]$ denotes the
power series expansion of a rational function $f$ in the region
$|z|>|w|$. We usually omit the symbol $\iota_{|z|>|w|}$ and just
write $(z-w)^{-1}$ to mean the expansion in the region $|z|>|w|$,
and write $-(w-z)^{-1}$ to mean the expansion in $|w|>|z|$. It is
easy to check that $a(w)\circ_n b(w)$ above is a well-defined
element of $QO(V)$.

The nonnegative circle products are connected through the {\it
operator product expansion} (OPE) formula.
For $a,b\in QO(V)$, we have \begin{equation}\label{opeform} a(z)b(w)=\sum_{n\geq 0}a(w)\circ_n
b(w)~(z-w)^{-n-1}+:a(z)b(w):\ ,\end{equation} which is often written as
$a(z)b(w)\sim\sum_{n\geq 0}a(w)\circ_n b(w)~(z-w)^{-n-1}$, where
$\sim$ means equal modulo the term $$
:a(z)b(w):\ =a(z)_-b(w)\ +\ (-1)^{|a||b|} b(w)a(z)_+.$$ Here
$a(z)_-=\sum_{n<0}a(n)z^{-n-1}$ and $a(z)_+=\sum_{n\geq
0}a(n)z^{-n-1}$. Note that $:a(w)b(w):$ is a well-defined element of
$QO(V)$. It is called the {\it Wick product} of $a$ and $b$, and it
coincides with $a\circ_{-1}b$. The other negative circle products
are related to this by
$$ n!~a(z)\circ_{-n-1}b(z)=\ :(\partial^n a(z))b(z):\ ,$$
where $\partial$ denotes the formal differentiation operator
$\frac{d}{dz}$. For $a_1(z),\dots ,a_k(z)\in QO(V)$, the $k$-fold
iterated Wick product is defined to be
\begin{equation}\label{iteratedwick} :a_1(z)a_2(z)\cdots a_k(z):\ =\ :a_1(z)b(z):~,\end{equation}
where $b(z)=\ :a_2(z)\cdots a_k(z):\ $. We often omit the formal variable $z$ when no confusion can arise.

The set $QO(V)$ is a nonassociative algebra with the operations
$\circ_n$, which satisfy $1\circ_n a=\delta_{n,-1}a$ for
all $n$, and $a\circ_n 1=\delta_{n,-1}a$ for $n\geq -1$. In particular, $1$ behaves as a unit with respect to $\circ_{-1}$. A linear subspace $\cA\subset QO(V)$ containing $1$ which is closed under the circle products will be called a {\it quantum operator algebra} (QOA). Note that $\cA$ is closed under $\partial$
since $\partial a=a\circ_{-2}1$. Many formal algebraic
notions are immediately clear: a homomorphism is just a linear
map that sends $1$ to $1$ and preserves all circle products; a module over $\cA$ is a
vector space $M$ equipped with a homomorphism $\cA\rightarrow
QO(M)$, etc. A subset $S=\{a_i|\ i\in I\}$ of $\cA$ is said to generate $\cA$ if every element $a\in\cA$ can be written as a linear combination of nonassociative words in the letters $a_i$, $\circ_n$, for
$i\in I$ and $n\in\mathbb{Z}$. We say that $S$ {\it strongly generates} $\cA$ if every $a\in\cA$ can be written as a linear combination of words in the letters $a_i$, $\circ_n$ for $n<0$. Equivalently, $\cA$ is spanned by the collection $\{ :\partial^{k_1} a_{i_1}(z)\cdots \partial^{k_m} a_{i_m}(z):| ~i_1,\dots,i_m \in I,~ k_1,\dots,k_m \geq 0\}$.

We say that $a,b\in QO(V)$ {\it quantum commute} if $(z-w)^N
[a(z),b(w)]=0$ for some $N\geq 0$. Here $[,]$ denotes the super bracket. This condition implies that $a\circ_n b = 0$ for $n\geq N$, so (\ref{opeform}) becomes a finite sum. A {\it commutative quantum operator algebra} (CQOA) is a QOA whose elements pairwise quantum commute. Finally, the notion of a CQOA is equivalent to the notion of a vertex algebra. Every CQOA $\cA$ is itself a faithful $\cA$-module, called the {\it left regular
module}. Define
$$\rho:\cA\rightarrow QO(\cA),\ \ \ \ a\mapsto\hat a,\ \ \ \ \hat
a(\zeta)b=\sum_{n\in\mathbb{Z}} (a\circ_n b)~\zeta^{-n-1}.$$ Then $\rho$ is an injective QOA homomorphism,
and the quadruple of structures $(\cA,\rho,1,\partial)$ is a vertex
algebra in the sense of \cite{FLM}. Conversely, if $(V,Y,{\bf 1},D)$ is
a vertex algebra, the collection $Y(V)\subset QO(V)$ is a
CQOA. {\it We will refer to a CQOA simply as a
vertex algebra throughout the rest of this paper}.

The following are useful identities that measure the nonassociativity and noncommutativity of the Wick product, and the failure of the positive circle products to be derivations of the
Wick product. Let $a,b,c$ be vertex operators in some vertex algebra $\cA$, and let $n > 0$. Then
\begin{equation}\label{vaidi} :(:ab:)c:-:abc:=\sum_{k\geq0}{1\over(k+1)!}\left(:(\partial^{k+1}a)(b\circ_k
c): +(-1)^{|a||b|}:(\partial^{k+1}b)(a\circ_k c):\right),\end{equation}
\begin{equation}\label{vaidii} :ab:-(-1)^{|a||b|}:ba:=\sum_{k\geq0}{(-1)^k\over(k+1)!}\partial^{k+1}(a\circ_kb),\end{equation}
\begin{equation}\label{vaidiii} a\circ_n(:bc:)-:(a\circ_nb)c:-(-1)^{|a||b|}:b(a\circ_nc):= \sum_{k=1}^n\left(\begin{matrix} n\cr k\end{matrix} \right)(a\circ_{n-k}b)\circ_{k-1}c, \end{equation}
\begin{equation}\label{vaidiv} (:ab:)\circ_n
c=\sum_{k\geq0}{1\over k!}:(\partial^ka)(b\circ_{n+k}c):
+(-1)^{|a||b|}\sum_{k\geq0}b\circ_{n-k-1}(a\circ_k c) .\end{equation}

\section{Category $\cR$}
In \cite{LL} we considered a certain category $\cR$ of vertex algebras, together with a functor from $\cR$ to the category of supercommutative rings.

\begin{defn}Let $\cR$ be the category of vertex algebras $\cA$ equipped with a $\mathbb{Z}_{\geq 0}$-filtration
\begin{equation}\label{goodi} \cA_{(0)}\subset\cA_{(1)}\subset\cA_{(2)}\subset \cdots,\ \ \ \cA = \bigcup_{k\geq 0}
\cA_{(k)}\end{equation} such that $\cA_{(0)} = \mathbb{C}$, and for all
$a\in \cA_{(k)}$, $b\in\cA_{(l)}$, we have
\begin{equation}\label{goodii} a\circ_n b\in\cA_{(k+l)},\ \ \ \text{for}\
n<0,\end{equation}
\begin{equation}\label{goodiii} a\circ_n b\in\cA_{(k+l-1)},\ \ \ \text{for}\
n\geq 0.\end{equation}
Elements $a(z)\in\cA_{(d)}\setminus \cA_{(d-1)}$ are said to have
degree $d$, and morphisms in $\cR$ are vertex
algebra homomorphisms that preserve the filtration.\end{defn}

Filtrations on vertex algebras satisfying (\ref{goodii})-(\ref{goodiii}) were introduced in \cite{LiII} and are known as {\it good increasing filtrations}. Setting $\cA_{(-1)} = \{0\}$, the associated graded object $gr(\cA) = \bigoplus_{k\geq 0}\cA_{(k)}/\cA_{(k-1)}$ is a
$\mathbb{Z}_{\geq 0}$-graded associative, supercommutative algebra with a
unit $1$ under a product induced by the Wick product on $\cA$. In general, there is no natural linear map $\cA\ra gr (\cA)$, but for each $r\geq 1$ we have the projection \begin{equation}\label{proj} \phi_r: \cA_{(r)} \ra \cA_{(r)}/\cA_{(r-1)}\subset gr(\cA).\end{equation} 
Moreover, $gr(\cA)$ has a derivation $\partial$ of degree zero
(induced by the operator $\partial = \frac{d}{dz}$ on $\cA$), and
for each $a\in\cA_{(d)}$ and $n\geq 0$, the operator $a\circ_n$ on $\cA$
induces a derivation of degree $d-k$ on $gr(\cA)$, which we also denote by $a\circ_n$. Here $$k  = sup \{ j\geq 1|~ \cA_{(r)}\circ_n \cA_{(s)}\subset \cA_{(r+s-j)}~\forall r,s,n\geq 0\},$$ as in \cite{LL}. Finally, these derivations give $gr(\cA)$ the structure of a vertex Poisson algebra.

The assignment $\cA\mapsto gr(\cA)$ is a functor from $\cR$ to the category of $\mathbb{Z}_{\geq 0}$-graded supercommutative rings with a differential $\partial$ of degree 0, which we will call $\partial$-rings. A $\partial$-ring is the same thing as an {\it abelian} vertex algebra, that is, a vertex algebra $\cV$ in which $[a(z),b(w)] = 0$ for all $a,b\in\cV$. A $\partial$-ring $A$ is said to be generated by a subset $\{a_i|~i\in I\}$ if $\{\partial^k a_i|~i\in I, k\geq 0\}$ generates $A$ as a graded ring. The key feature of $\cR$ is the following reconstruction property \cite{LL}:

\begin{lemma} \label{recon} Let $\cA$ be a vertex algebra in $\cR$ and let $\{a_i|~i\in I\}$ be a set of generators for $gr(\cA)$ as a $\partial$-ring, where $a_i$ is homogeneous of degree $d_i$. If $a_i(z)\in\cA_{(d_i)}$ are vertex operators such that $\phi_{d_i}(a_i(z)) = a_i$, then $\cA$ is strongly generated as a vertex algebra by $\{a_i(z)|~i\in I\}$.\end{lemma}

As shown in \cite{LI}, there is a similar reconstruction property for kernels of surjective morphisms in $\cR$. Let $f:\cA\ra \cB$ be a morphism in $\cR$ with kernel $\cJ$, such that $f$ maps $\cA_{(k)}$ onto $\cB_{(k)}$ for all $k\geq 0$. The kernel $J$ of the induced map $gr(f): gr(\cA)\ra gr(\cB)$ is a homogeneous $\partial$-ideal (i.e., $\partial J \subset J$). A set $\{a_i|~i\in I\}$ such that $a_i$ is homogeneous of degree $d_i$ is said to generate $J$ as a $\partial$-ideal if $\{\partial^k a_i|~i\in I,~k\geq 0\}$ generates $J$ as an ideal.

\begin{lemma} \label{idealrecon} Let $\{a_i| i\in I\}$ be a generating set for $J$ as a $\partial$-ideal, where $a_i$ is homogeneous of degree $d_i$. Then there exist vertex operators $a_i(z)\in \cA_{(d_i)}$ with $\phi_{d_i}(a_i(z)) = a_i$, such that $\{a_i(z)|~i\in I\}$ generates $\cJ$ as a vertex algebra ideal.\end{lemma}

\section{The structure of $\cH(n)^{O(n)}$} \label{secortho}

The ring of Laurent polynomials $\mathbb{C}[t,t^{-1}]$ may be regarded as an abelian Lie algebra. It has a central extension $\gh = \mathbb{C}[t,t^{-1}]\oplus \mathbb{C}\kappa$ with bracket $[t^n,t^m] = n \delta_{n+m,0} \kappa$, and $\mathbb{Z}$-gradation $deg(t^n) = n$, $deg(\kappa) = 0$. Let $\gh_{\geq 0} = \oplus_{n\geq 0} \gh_n$, and let $C$ be the one-dimensional $\gh_{\geq 0}$-module on which $t^n$ acts trivially for $n\geq 0$, and $\kappa$ acts by the identity. Define $V = U(\gh)\otimes_{U(\gh_{\geq 0})} C$, and let $\alpha(n)\in End(V)$ be the linear operator representing $t^n$ on $V$. Define $\alpha(z) = \sum_{n\in\mathbb{Z}} \alpha(n) z^{-n-1}$, which is easily seen to lie in $QO(V)$ and satisfy the OPE relation $$\alpha(z)\alpha(w)\sim (z-w)^{-2}.$$ The vertex algebra $\cH$ generated by $\alpha$ is known as the {\it Heisenberg vertex algebra}. The rank $n$ Heisenberg algebra $\cH(n)$ is just the tensor product of $n$ copies of $\cH$, with generators $\alpha^1,\dots,\alpha^n$. There is a natural conformal structure of central charge $n$ on $\cH(n)$, with Virasoro element \begin{equation}\label{virasoro} L(z) = \frac{1}{2} \sum_{i=1}^n :\alpha^i(z) \alpha^i(z):,\end{equation}  under which each $\alpha^i$ is primary of weight one. The full automorphism group of $\cH(n)$ preserving $L(z)$ is easily seen to be the orthogonal group $O(n)$, which acts linearly on the vector space $U$ spanned by $\{\alpha^1,\dots,\alpha^n\}$. First, any conformal automorphism $\phi$ of $\cH(n)$ must lie in $GL(U)$ by weight considerations. Moreover, since $\alpha^i \circ_1 \alpha^j = \phi(\alpha^i)\circ_1 \phi(\alpha^j) = \delta_{i,j}$, $\phi$ must preserve the pairing $\bra,\ket$ on $U$ defined by $\bra \alpha^i,\alpha^j\ket = \delta_{i,j}$.

We define a good increasing filtration on $\cH(n)$ as follows: $\cH(n)_{(r)}$ is spanned by the set \begin{equation}\label{goodsv} \{:\partial^{k_1} \alpha^{i_1}  \cdots \partial^{k_s}\alpha^{i_s} :| ~ k_j\geq 0,~ s \leq r\}.\end{equation} Then $\cH(n)\cong gr(\cH(n))$ as linear spaces, and as commutative algebras we have
\begin{equation}\label{structureofgrs} gr(\cH(n))\cong Sym \bigoplus_{k\geq 0} V_k.\end{equation} In this notation, $V_k$ is the linear span of  $\{\alpha^{i}_k |~ i=1,\dots,n\}$, where $\alpha^i_k$ is the image of $\partial^k \alpha^i(z)$ in $gr(\cH(n))$ under the projection $\phi_1: \cH(n)_{(1)}\ra \cH(n)_{(1)}/\cH(n)_{(0)}\subset gr(\cH(n))$. The action of $O(n)$ on $\cH(n)$ preserves this filtration, and induces an action of $O(n)$ on $gr(\cH(n))$ by algebra automorphisms. For all $k\geq 0$ we have isomorphisms of $O(n)$-modules $V_k\cong \mathbb{C}^n$. Finally, for any reductive subgroup $G\subset O(n)$, $\cH(n)^G\cong gr(\cH(n)^G)$ as linear spaces, and \begin{equation}\label{structureofgrsinv} gr(\cH(n)^G )\cong (gr(\cH(n))^G \cong (Sym \bigoplus_{k\geq 0} V_k)^G \end{equation}  as commutative algebras. In the case $G=O(n)$, the following classical theorem of Weyl \cite{W} describes the generators and relations of the ring $(Sym \bigoplus_{k\geq 0} V_k)^{O(n)}$:

\begin{thm}\label{weylfft} (Weyl) For $k\geq 0$, let $V_k$ be the copy of the standard $O(n)$-module $\mathbb{C}^n$ with orthonormal basis $\{x_{i,k}| ~i=1,\dots,n\}$. The invariant ring $(Sym \bigoplus_{k\geq 0} V_k )^{O(n)}$ is generated by the quadratics \begin{equation}\label{weylgenerators} q_{a,b} = \sum_{i=1}^n x_{i,a} x_{i,b},\ \ \ \ \ \  0\leq a\leq b.\end{equation} For $a>b$, define $q_{a,b} = q_{b,a}$, and let $\{Q_{a,b}|\ a,b\geq 0\}$ be commuting indeterminates satisfying $Q_{a,b} = Q_{b,a}$ and no other algebraic relations. The kernel $I_n$ of the homomorphism $\mathbb{C}[Q_{a,b}]\ra (Sym \bigoplus_{k\geq 0} V_k)^{O(n)}$ sending $Q_{a,b}\mapsto q_{a,b}$ is generated by the $(n+1)\times (n+1)$ determinants \begin{equation}\label{weylrel} d_{I,J} = \det \left[\begin{matrix} Q_{i_0,j_0} & \cdots & Q_{i_0,j_n} \cr  \vdots  & & \vdots  \cr  Q_{i_n,j_0}  & \cdots & Q_{i_n,j_n} \end{matrix} \right].\end{equation} In this notation, $I=(i_0,\dots, i_{n})$ and $J = (j_0,\dots, j_{n})$ are lists of integers satisfying \begin{equation}\label{ijineq} 0\leq i_0<\cdots <i_n,\ \ \ \ \ \  0\leq j_0<\cdots <j_n.\end{equation} Since $Q_{a,b} = Q_{b,a}$, it is clear that $d_{I,J} = d_{J,I}$. \end{thm}

Under the projection $$\phi_2: (\cH(n)^{O(n)})_{(2)}\ra (\cH(n)^{O(n)})_{(2)}/(\cH(n)^{O(n)})_{(1)}\subset gr(\cH(n)^{O(n)}) \cong (Sym \bigoplus_{k\geq 0} V_k)^{O(n)},$$ the generators $q_{a,b}$ of $(Sym \bigoplus_{k\geq 0} V_k)^{O(n)}$ correspond to vertex operators $\omega_{a,b}$ given by \begin{equation}\label{omegagen} \omega_{a,b} = \sum_{i=1}^n :\partial^a \alpha^i \partial^b \alpha^i:,\ \ \ \ \ \ 0\leq a\leq b.\end{equation}  By Lemma \ref{recon}, the set $\{\omega_{a,b}|~0\leq a \leq b\}$ is a strong generating set for $\cH(n)^{O(n)}$. Note that $\omega_{0,0} = 2L$, where $L$ is the Virasoro element (\ref{virasoro}). The subspace $(\cH(n)^{O(n)})_{(2)}$ of degree at most 2 has a basis $\{1\} \cup \{\omega_{a,b}\}$, and for all $n\geq 0$, the operators $\omega_{a,b}\circ_n$ preserve this vector space. It follows that every term in the OPE formula for $\omega_{a,b}(z) \omega_{c,d}(w)$ is a linear combination of these generators, so they form a Lie conformal algebra. We calculate that for $a,b,c\geq 0$ and $0\leq m \leq a+b+c+1$, 
\begin{equation}\label{bcalc} \omega_{a,b} \circ_m \partial^c \alpha^i = \lambda_{a,b,c,m} \partial^{a+b+c+1-m} \alpha^i\end{equation} where $$\lambda_{a,b,c,m}  = (-1)^b \frac{(b+c+1)!}{(b+c+1-m)!} + (-1)^a \frac{(a+c+1)!}{(a+c+1-m)!}.$$ It follows that for $m\leq a+b+c+1$ we have \begin{equation}\label{opeformula} \omega_{a,b}\circ_m \omega_{c,d} = \lambda_{a,b,c,m} \omega_{a+b+c+1-m,d} + \lambda_{a,b,d,m}\omega_{c,a+b+d+1-m}.\end{equation} In fact, there is a somewhat more economical set of strong generators for $\cH(n)^{O(n)}$. For each $m\geq 0$, let $A_m$ denote the vector space spanned by $\{\omega_{a,b}|~ a+b = m\}$, which is homogeneous of weight $m+2$. Clearly $\text{dim}(A_{2m}) = m+1 = \text{dim}(A_{2m+1})$ for $m\geq 0$. Moreover, $\partial(A_m)\subset A_{m+1}$, and we have \begin{equation}\label{deca} \text{dim} \big(A_{2m} / \partial(A_{2m-1})\big) = 1,\ \ \ \ \ \ \text{dim} \big(A_{2m+1} / \partial(A_{2m})\big) = 0.\end{equation} For $m\geq 0$, define \begin{equation}\label{defofj} j^{2m} = \omega_{0,2m},\end{equation} which is clearly not a total derivative. Hence $A_{2m}$ has a decomposition \begin{equation}\label{decompofa} A_{2m} = \partial (A_{2m-1})\oplus \bra j^{2m}\ket =  \partial^2 (A_{2m-2})\oplus \bra j^{2m}\ket ,\end{equation} where $\bra j^{2m}\ket$ is the linear span of $j^{2m}$. Similarly, $A_{2m+1}$ has a decomposition \begin{equation}\label{decompofai} A_{2m+1} = \partial^2(A_{2m-1})\oplus \bra \partial j^{2m}\ket =  \partial^3 (A_{2m-2})\oplus \bra \partial j^{2m}\ket.\end{equation}  It is easy to see that $\{\partial^{2i} j^{2m-2i}|~ 0\leq i\leq m\}$ and $\{\partial^{2i+1} j^{2m-2i}|\ 0\leq i\leq m\}$ are bases of $A_{2m}$ and $A_{2m+1}$, respectively. Hence each $\omega_{a,b}\in A_{2m}$ and $\omega_{c,d}\in A_{2m+1}$ can be expressed uniquely in the form \begin{equation}\label{lincomb} \omega_{a,b} =\sum_{i=0}^m \lambda_i \partial^{2i}j^{2m-2i},\ \ \ \ \ \  \omega_{c,d} =\sum_{i=0}^m \mu_i \partial^{2i+1}j^{2m-2i}\end{equation} for constants $\lambda_i,\mu_i$, $i=0,\dots,m$. Hence $\{j^{2m}|\ m\geq 0\}$ is an alternative strong generating set for $\cH(n)^{O(n)}$, and it will be convenient to pass back and forth between the sets $\{j^{2m}|\ m\geq 0\}$ and $\{\omega_{a,b}|\ 0\leq a\leq b\}$.

\begin{lemma} \label{ordfingen} $\cH(n)^{O(n)}$ is generated as a vertex algebra by $j^0$ and $j^2$.\end{lemma}
\begin{proof} Let $\cJ$ denote the vertex subalgebra of $\cH(n)^{O(n)}$ generated by $j^0$ and $j^2$. We need to show that $j^{2m}\in\cJ$ for all $m\geq 2$. Specializing (\ref{opeformula}) shows that \begin{equation}\label{specialope} j^2\circ_1 j^{2k} = 4 \omega_{2,2k} + (4+4k) j^{2k+2}.\end{equation} Moreover, it is easy to check that $\omega_{2,2k} \equiv - j^{2k+2}$ modulo second derivatives. It follows that $j^2\circ_1j^{2k} \equiv (3+4k)j^{2k+2}$ modulo a linear combination of elements of the form $\partial^{2i} j^{2k+2-2i}$ for $1\leq i \leq k+1$. The claim then follows by induction on $k$. \end{proof}

Consider the category of all vertex algebras with generators $\{J^{2m}|~m\geq 0\}$, which satisfy the same OPE relations as the generators $\{j^{2m}|~m\geq 0\}$ of $\cH(n)^{O(n)}$. Since the vector space with basis $\{1\}\cup \{\partial^lj^{2m}|~l,m\geq 0\}$ is closed under all nonnegative circle products, it forms a Lie conformal algebra. By Theorem 7.12 of \cite{BK}, this category possesses a universal object $\cV_n$, which is {\it freely} generated by $\{J^{2m}|\ m\geq 0\}$. In other words, there are no nontrivial normally ordered polynomial relations among the generators and their derivatives in $\cV_n$. Then $\cH(n)^{O(n)}$ is a quotient of $\cV_n$ by an ideal $\cI_n$, and since $\cH(n)^{O(n)}$ is a simple vertex algebra, $\cI_n$ is a maximal ideal. Let $\pi_n: \cV_n\ra \cH(n)^{O(n)}$ denote the quotient map, which sends $J^{2m}\mapsto j^{2m}$. Using the formula (\ref{lincomb}), which holds in $\cH(n)^{O(n)}$ for all $n$, we can define an alternative strong generating set $\{\Omega_{a,b}| ~0\leq a\leq b\}$ for $\cV_n$ by the same formula: for $a+b = 2m$ and $c+d = 2m+1$,
$$\Omega_{a,b} =\sum_{i=0}^m \lambda_i \partial^{2i}J^{2m-2i},\ \ \ \ \ \  \Omega_{c,d} =\sum_{i=0}^m \mu_i \partial^{2i+1}J^{2m-2i}.$$ Clearly $\pi_n(\Omega_{a,b}) = \omega_{a,b}$. We will use the same notation $A_m$ to denote the linear span of $\{\Omega_{a,b}|~ a+b = m\}$, when no confusion can arise. Finally, $\cV_{n}$ has a good increasing filtration in which $(\cV_n)_{(2k)}$ is spanned by iterated Wick products of the generators $J^{2m}$ and their derivatives, of length at most $k$, and $(\cV_n)_{(2k+1)} = (\cV_n)_{(2k)}$. Equipped with this filtration, $\cV_n$ lies in the category $\cR$, and $\pi_n$ is a morphism in $\cR$.

\begin{remark}Since the vertex operators $J^{2m}$ satisfy the same OPE relations as the $j^{2m}$, $\cV_n$ is also generated as a vertex algebra by $J^0$ and $J^2$.\end{remark}

Recall the variables $Q_{a,b}$ and $q_{a,b}$ appearing in Theorem \ref{weylfft}. Since $\cV_n$ is freely generated by $\{J^{2m}|\ m\geq 0\}$, and $\{\Omega_{a,b}|\ 0\leq a\leq b\}$ and $\{\partial^k J^{2m}|\ k,m\geq 0\}$ form bases for the same space, we may identify $gr(\cV_n)$ with $\mathbb{C}[Q_{a,b}]$, and we identify $gr(\cH(n)^{O(n)})$ with $\mathbb{C}[q_{a,b}]/I_n$. Under this identification, $gr(\pi_{n}): gr(\cV_n) \ra gr(\cH(n)^{O(n)})$ is just the quotient map sending $Q_{a,b}\mapsto q_{a,b}$. Clearly the projection $\pi_{n}: \cV_n\ra \cH(n)^{O(n)}$ maps each filtered piece $(\cV_n)_{(k)}$ onto $(\cH(n)^{O(n)})_{(k)}$, so the hypotheses of Lemma \ref{idealrecon} are satisfied. Since $I_{n} = Ker (gr(\pi_{n}))$ is generated by the determinants $d_{I,J}$, we can apply Lemma \ref{idealrecon} to find vertex operators $D_{I,J}\in (\cV_{n})_{(2n+2)}$ satisfying $\phi_{2n+2}(D_{I,J}) = d_{I,J}$, such that $\{D_{I,J}\}$ generates $\cI_{n}$. Since $\Omega_{a,b}$ has weight $a+b+2$, \begin{equation}\label{wtod} wt(D_{I,J}) = |I| + |J| +2n+2,\ \ \ \ \ \ |I| =\sum_{a=0}^{n} i_a,\ \ \ \ \ \ |J| =\sum_{a=0}^{n} j_a.\end{equation}

In general, the vertex operators $a_i(z)$ furnished by Lemma \ref{idealrecon} satisfying $\phi_{d_i}(a_i(z)) = a_i$ which generate $\cI$ are not unique. However, in our case, $D_{I,J}$ is uniquely determined by the conditions \begin{equation}\label{uniquedij} \phi_{2n+2}(D_{I,J}) = d_{I,J},\ \ \ \ \ \ \pi_{n}(D_{I,J}) = 0.\end{equation} If $D'_{I,J}$ is another vertex operator satisfying (\ref{uniquedij}), $D_{I,J} - D'_{I,J}$ lies in $(\cV_n)_{(2n)} \cap \cI_{n}$, and since there are no relations in $\cH(n)^{O(n)}$ of degree less than $2n+2$, we have $D_{I,J} - D'_{I,J}=0$. There is a distinguished element $D_0 = D_{(0,\dots,n),(0,\dots,n)}$ which is the unique element of $\cI_n$ of minimal weight $n^2+3n+2$. It is annihilated by all operators $J^{2m}(k)$ for $k>2m+1$, which lower the weight by $k-2m-1$.

Given a homogeneous polynomial $p\in gr(\cV_{n})\cong \mathbb{C}[Q_{a,b}|~0\leq a\leq b]$ of degree $k$ in the variables $Q_{a,b}$, a {\it normal ordering} of $p$ will be a choice of normally ordered polynomial $P\in (\cV_{n})_{(2k)}$, obtained by replacing $Q_{a,b}$ by $\Omega_{a,b}$, and replacing ordinary products with iterated Wick products of the form (\ref{iteratedwick}). Of course $P$ is not unique, but for any choice of $P$ we have $\phi_{2k}(P) = p$, where $\phi_{2k}: (\cV_{n})_{(2k)} \ra (\cV_{n})_{(2k)} /(\cV_{n})_{(2k-1)} \subset gr(\cV_{n})$ is the usual projection. For the rest of this section, $D^{2k}$, $E^{2k}$, $F^{2k}$, etc., will denote elements of $(\cV_{n})_{(2k)}$ which are homogeneous, normally ordered polynomials of degree $k$ in the vertex operators $\Omega_{a,b}$.

Let $D_{I,J}^{2n+2}\in (\cV_{n})_{(2n+2)}$ be some normal ordering of $d_{I,J}$. Then $$\pi_{n}(D_{I,J}^{2n+2}) \in (\cH(n)^{O(n)})_{(2n)},$$ and $\phi_{2n}(\pi_n(D_{I,J}^{2n+2})) \in gr(\mathcal{H}(n)^{O(n)})$ can be expressed uniquely as a polynomial of degree $n$ in the variables $q_{a,b}$. Choose some normal ordering of the corresponding polynomial in the variables $\Omega_{a,b}$, and call this vertex operator $-D^{2n}_{I,J}$. Then $D^{2n+2}_{I,J} + D^{2n}_{I,J}$ has the property that $\pi_{n}(D^{2n+2}_{I,J} + D^{2n}_{I,J})\in (\cH(n)^{O(n)})_{(2n-2)}.$ Continuing this process, we arrive at a vertex operator $\sum_{k=1}^{n+1} D^{2k}_{I,J}$ in the kernel of $\pi_{n}$. We must have 
\begin{equation}\label{decompofd} D_{I,J} = \sum_{k=1}^{n+1}D^{2k}_{I,J},\end{equation} since $D_{I,J}$ is uniquely characterized by (\ref{uniquedij}).

In this decomposition, the term $D^2_{I,J}$ lies in the space $A_m$ spanned by $\{\Omega_{a,b}|~a+b=m\}$, for $m = |I| +|J|+2n$. Recall that for $m$ even, $A_m  = \partial^2 A_{m-2} \oplus \bra J^{m}\ket$, and for $m$ odd, $A_m  = \partial^3 A_{m-3} \oplus \bra \partial J^{m-1}\ket$. For $m$ even (respectively odd), define $pr_m: A_m\ra \bra J^m\ket$ (respectively $pr_m: A_m\ra \bra \partial J^{m-1}\ket$) to be the projection onto the second term. Define the {\it remainder} \begin{equation}\label{defofrij} R_{I,J} = pr_m(D^2_{I,J}).\end{equation}

\begin{lemma} \label{uniquenessofr} Given $D_{I,J}\in\cI_n$ as above, suppose that $D_{I,J} = \sum_{k=1}^{n+1} D^{2k}_{I,J}$ and $D_{I,J} = \sum_{k=1}^{n+1} \tilde{D}^{2k}_{I,J}$ are two different decompositions of $D_{I,J}$ of the form (\ref{decompofd}). Then $$D^2_{I,J} - \tilde{D}^2_{I,J} \in \partial^2 (A_{m-2}),$$ where $m = |I| +|J| +2 n$. In particular, $R_{I,J}$ is independent of the choice of decomposition of $D_{I,J}$.\end{lemma}

\begin{proof} This is analogous to Corollary 4.8 of \cite{LI}, and the proof is almost the same. First, we claim that for all $j,k,l,m\geq 0$, $\Omega_{j,k}\circ_0\Omega_{l,m}$ is a total derivative. In view of the decomposition (\ref{deca}) which holds in $\cV_n$ as well as $\cH(n)^{O(n)}$, it suffices to show that $J^{2k}\circ_0 J^{2l}$ is a total derivative for all $k,l$. This is clear because $J^{2k}\circ_0 J^{2l}$ lies in $A_m$ for $m=2k+2l+1$, and $A_m = \partial(A_{m-1})$.

Next, let $\mu = ~:a_1 \cdots a_m:$ be a normally ordered monomial in $(\cV_{n})_{(2m)}$, where each $a_i$ is one of the generators $\Omega_{a,b}$. Let $\tilde{\mu} = :a_{i_1} \cdots a_{i_m}:$, where $(i_1,\dots, i_m)$ is some permutation of $(1,\dots,m)$. We claim that for any decomposition
$\mu - \tilde{\mu} = \sum_{k=1}^{m-1} E^{2k}$ of the difference $\mu - \tilde{\mu}\in (\cV_{n})_{(2m-2)}$, the term $E^2$ is a second derivative. To prove this statement, we proceed by induction on $m$. For $m=1$, there is nothing to prove since $\mu - \tilde{\mu} =0$. For $m=2$, and $\mu = ~:\Omega_{a,b}\Omega_{c,d}:$, we have \begin{equation}\label{reari} \mu - \tilde{\mu} = ~:\Omega_{a,b}\Omega_{c,d}: ~- ~: \Omega_{c,d}\Omega_{a,b}:~ = \sum_{i\geq 0} \frac{(-1)^i}{(i+1)!}\partial^{i+1}(\Omega_{a,b}\circ_{i} \Omega_{c,d}),\end{equation} by (\ref{vaidii}). Since $\Omega_{a,b}\circ_{0} \Omega_{c,d}$ is already a total derivative, it follows that $\mu -\tilde{\mu}$ is a second derivative, as claimed. Next, we assume the result for $r\leq m-1$. Since the permutation group on $m$ letters is generated by the transpositions $(i,i+1)$ for $i=1,\dots,m-1$, we may assume without loss of generality that $$\tilde{\mu} = ~:a_1 \cdots a_{i-1} a_{i+1} a_i a_{i+2} \cdots a_m:.$$ If $i>1$, we have $\mu - \tilde{\mu} = ~:a_1 \cdots a_{i-1} f:$, where $f =~ :a_i \cdots a_m:~ -~: a_{i+1}a_i a_{i+2} \cdots a_m$, which lies in $(\cM_{-n})_{(2m-2i+2)}$. Since each term of $f$ has degree at least 2, it follows that $\mu - \tilde{\mu}$ can be expressed in the form $\sum_{k=i}^{m-1}E^{2k}$. Since $i>1$, there is no term of degree $2$. Given any rearrangement $\mu - \tilde{\mu}=\sum_{k=1}^{m-1}F^{2k}$, it follows from our inductive hypothesis that the term $F^2$ is a second derivative.

Suppose next that $i=1$, so that $\tilde{\mu} = ~:a_2 a_1 a_3 \cdots a_m:$. Define $$\nu = \ :(:a_1 a_2:) a_3 \cdots a_m:,\ \ \ \ \ \ \tilde{\nu} = \ :(:a_2 a_1:) a_3 \cdots a_m:,$$ and note that $\nu  - \tilde{\nu} = \ : (:a_1 a_2: -: a_2 a_1: )f:$, where $f = \ :a_3 \cdots a_m:$. By (\ref{reari}), $:a_1 a_2: - : a_2 a_1:$ is homogeneous of degree 2, so $\nu  - \tilde{\nu}$ is a linear combination of monomials of degree $2m-2$. By inductive assumption, any rearrangement $\nu - \tilde{\nu} =\sum_{k=1}^{m-1}F^{2k}$ has the property that $F^2$ is a second derivative.

Next, by (\ref{vaidi}), we have \begin{equation}\label{diffi} \mu - \nu = - \sum_{k\geq 0} \frac{1}{(k+1)!} \bigg( :(\partial^{k+1} a_1) (a_2 \circ_k f): + : (\partial^{k+1} a_2)(a_1\circ_k f):\bigg).\end{equation} Since the operators $\circ_k$ for $k\geq 0$ are homogeneous of degree $-2$, each term appearing in (\ref{diffi}) has degree at most $2m-2$. Moreover, $$deg \big(:(\partial^{k+1} a_1) (a_2\circ_k f):\big) = 2+ deg(a_2 \circ_k f),\ \ \ deg \big(:(\partial^{k+1} a_2) (a_1\circ_k f):\big) = 2+ deg(a_1 \circ_k f),$$ so the only way to obtain terms of degree $2$ is for $a_2\circ_k f$ or $a_1\circ_k f$ to be a scalar. This can only happen if $k>0$, in which case we obtain either $\partial^{k+1} a_1$ or $\partial^{k+1} a_2$, which are second derivatives. By inductive assumption, any rearrangement of $\mu - \nu$ can contain only second derivatives in degree 2. Similarly, $\tilde{\mu} - \tilde{\nu}$ has degree at most $2m-2$, and any rearrangement of $\tilde{\mu} - \tilde{\nu}$ can only contain second derivatives in degree 2. Since $\mu - \tilde{\mu} = (\mu - \nu) + (\nu - \tilde{\nu}) + (\tilde{\nu} - \tilde{\mu})$, the claim follows. 

An immediate consequence is the following statement. Let $E\in (\cV_{n})_{(2m)}$ be a vertex operator of degree $2m$, and choose a decomposition \begin{equation}\label{dcnew} E = \sum_{k=1}^m E^{2k},\end{equation} where $E^{2k}$ is a homogeneous, normally ordered polynomial of degree $k$ in the variables $\Omega_{a,b}$. If $E = \sum_{k=1}^m F^{2k}$ is any rearrangement of (\ref{dcnew}), i.e., another decomposition of $E$ of the same form, then $E^{2}-F^2$ is a second derivative. Finally, specializing this to the case $E=D_{I,J}$ proves the lemma. \end{proof}

\begin{lemma}Let $R_0$ denote the remainder of the element $D_0$. The condition $R_0 \neq 0$ is equivalent to the existence of a decoupling relation of the form $j^{n^2+3n} = P(j^0,j^2,\dots, j^{n^2+3n-2})$ in $\cH(n)^{O(n)}$.\end{lemma}

\begin{proof} Let $D_0 =  \sum_{k=1}^{n+1}D^{2k}_0$ be a decomposition of $D_0$ of the form (\ref{decompofd}). If $R_0\neq 0$, we have $D^2_0 = \lambda J^{n^2+3n} + \partial^2 \omega$ for some $\lambda \neq 0$ and some $\omega\in A_{n^2+3n-2}$. Applying the projection $\pi_{n}:\cV_{n}\ra \cH(n)^{O(n)}$, since $\pi_{n}(D_0)=0$ we obtain $$j^{n^2+3n}= -\frac{1}{\lambda}\big( \partial^2 \pi_{n}(\omega) + \sum_{k=2}^{n+1} \pi_{n}(D^{2k}_0) \big),$$ which is a decoupling relation of the desired form. The converse follows from the fact that $D_0$ is the unique element of the ideal $\cI_{n}$ of weight $n^2+3n+2$, up to scalar multiples. \end{proof}

\begin{conj} \label{mainconj} For all $n\geq 1$, the remainder $R_0$ is nonzero.\end{conj}

For $n=1$, it is easy to check that $R_0 = -\frac{5}{4} J^4$, and in the Appendix, we write down computer calculations that prove this conjecture in the cases $n=2$ and $n=3$. For $n=2$, we have $R_0 = \frac{149}{600} J^{10}$, and for $n=3$ we have $R_0 = -\frac{2419}{705600} J^{18}$. However, for $n>3$ it is difficult to calculate $R_0$ directly. Unlike the case of $\cW_{1+\infty,-n}$ where a similar remainder was shown to be nonzero in \cite{LI}, there seems to be no nice recursive structure that allows us to proceed by induction on $n$. Even in the case $\cW_{1+\infty,-n}$, we still lack a conceptual explanation for this phenomenon, and we expect that such an explanation will be necessary to prove our conjecture for $\cH(n)^{O(n)}$.

The next theorem shows that the strong finite generation of $\cH(n)^{O(n)}$ is an easy consequence of our conjecture.

\begin{thm} \label{stronggen} Suppose that Conjecture \ref{mainconj} holds. Then for all $r\geq \frac{1}{2}(n^2+3n)$, there exists a decoupling relation \begin{equation}\label{maindecoup} j^{2r} = Q_{2r}(j^0,j^2,\dots,j^{n^2+3n-2}),\end{equation} where $Q_{2r}$ is some normally ordered polynomial in $j^0,j^2,\dots,j^{n^2+3n-2}$, and their derivatives. It follows that $\{j^0,j^2,\dots,j^{n^2+3n-2}\}$ is a minimal strong generating set for $\cH(n)^{O(n)}$. \end{thm}

\begin{proof} The decoupling relation $j^{n^2+3n} = P(j^0,j^2,\dots,j^{n^2+3n-2})$ given by Conjecture \ref{mainconj} corresponds to an element $J^{n^2+3n} - P(J^0,J^2,\dots,J^{n^2+3n-2})\in \cI_{n}$. We need to show that for all $r\geq \frac{1}{2}(n^2+3n)$, there exists an element $J^{2r} - Q_{2r}(J^0,J^2,\dots,J^{n^2+3n-2})\in \cI_{n}$, so we assume inductively that $Q_{2r-2}$ exists. Choose a decomposition $$Q_{2r-2} = \sum_{k=1}^{d} Q^{2k}_{2r-2},$$ where $Q_{2r-2}^{2k}$ is a homogeneous normally ordered polynomial of degree $k$ in the vertex operators $J^0,\dots,J^{n^2+3n-2}$ and their derivatives. In particular, $$Q^2_{2r-2} = \sum_{i=0}^{\frac{1}{2}(n^2+3n-2)} c_i \partial^{2r-2i-2} J^{2i},$$ for constants $c_i$. We apply the operator $J^2 \circ_1$, which raises the weight by two. By (\ref{specialope}), we have $J^2 \circ_1 J^{2r-2} \equiv (4r-1)J^{2r}$ modulo second derivatives. Moreover, using (\ref{vaidiii}) and (\ref{specialope}), we see that $J^2 \circ_1 \big(\sum_{k=1}^{d} Q^{2k}_{2r-2}\big)$ can be expressed in the form $\sum_{k=1}^{d} E^{2k}$ where each $E^{2k}$ is a normally ordered polynomial in $J^0,\dots,J^{n^2+3n}$ and their derivatives. If $J^{n^2+3n}$ or its derivatives appear in $E^{2k}$, we can use the element $J^{n^2+3n} -P(J^0,\dots, J^{n^2+3n-2})$ in $\cI_{n}$ to eliminate the variable $J^{n^2+3n}$ and any of its derivatives, modulo $\cI_{n}$. Hence $J^2 \circ_1 \big(\sum_{k=1}^{d} Q^{2k}_{2r-2}\big)$ can be expressed modulo $\cI_{n}$ in the form $\sum_{k=1}^{d'} F^{2k}$, where $d'\geq d$, and $F^{2k}$ is a normally ordered polynomial in $J^0,\dots, J^{n^2+3n-2}$ and their derivatives. It follows that $$\frac{1}{4r-1} J^2 \circ_1 \big(J^{2r-2} - Q_{2r-2}(J^0,\dots, J^{n^2+3n-2})\big)$$ can be expressed as an element of $\cI_{n}$ of the desired form. \end{proof}

\section{Representation theory of $\cH(n)^{O(n)}$}
The basic tool in studying the representation theory of vertex algebras is the {\it Zhu functor}, which was introduced by Zhu in \cite{Z}. Given a vertex algebra $\cW$ with weight grading $\cW = \bigoplus_{n\in\mathbb{Z}} \cW_n$, this functor attaches to $\cW$ an associative algebra $A(\cW)$, together with a surjective linear map $\pi_{Zh}:\cW\ra A(\cW)$. For $a\in \cW_{m}$ and $b\in\cW$, define
\begin{equation}\label{defzhu} a*b = Res_z \bigg (a(z) \frac{(z+1)^{m}}{z}b\bigg),\end{equation} and extend $*$ by linearity to a bilinear operation $\cW\otimes \cW\ra \cW$. Let $O(\cW)$ denote the subspace of $\cW$ spanned by elements of the form \begin{equation}\label{zhuideal} a\circ b = Res_z \bigg (a(z) \frac{(z+1)^{m}}{z^2}b\bigg)\end{equation} where $a\in \cW_m$, and let $A(\cW)$ be the quotient $\cW/O(\cW)$, with projection $\pi_{Zh}:\cW\ra A(\cW)$. Then $O(\cW)$ is a two-sided ideal in $\cW$ under the product $*$, and $(A(\cW),*)$ is a unital, associative algebra. The assignment $\cW\mapsto A(\cW)$ is functorial, and if $\cI$ is a vertex algebra ideal of $\cW$, we have $A(\cW/\cI)\cong A(\cW)/ I$, where $I = \pi_{Zh}(\cI)$. A well-known formula asserts that for all $a\in \cW_m$ and $b\in \cW$, \begin{equation}\label{zhucomm} a*b - b*a \equiv Res_z (1+z)^{m -1} a(z) b \ \ \ \  \text{mod} \ O(\cW).\end{equation}

A $\mathbb{Z}_{\geq 0}$-graded module $M = \bigoplus_{n\geq 0} M_n$ over $\cW$ is called {\it admissible} if for every $a\in\cW_m$, $a(n) M_k \subset M_{m+k -n-1}$, for all $n\in\mathbb{Z}$. Given $a\in\cW_m$, the Fourier mode $a(m-1)$ acts on each $M_k$. The subspace $M_0$ is then a module over $A(\cW)$ with action $[a]\mapsto a(m-1) \in End(M_0)$. In fact, $M\mapsto M_0$ provides a one-to-one correspondence between irreducible, admissible $\cW$-modules and irreducible $A(\cW)$-modules. If $A(\cW)$ is a commutative algebra, all its irreducible modules are one-dimensional, and the corresponding $\cW$-modules $M = \bigoplus_{n\geq 0} M_n$ are cyclic and generated by any nonzero $v\in M_0$. Accordingly, we call such a module a {\it highest-weight module} for $\cW$, and we call $v$ a {\it highest-weight vector}.

Let $\cW$ be a vertex algebra which is strongly generated by a set of weight-homogeneous elements $\alpha_i$ of weights $w_i$, for $i$ in some index set $I$. Then $A(\cW)$ is generated by $\{ a_i = \pi_{Zh}(\alpha_i(z))|~i\in I\}$. Moreover, $A(\cW)$ inherits a filtration (but not a grading) by weight.

\begin{thm} \label{abelianzhu} For all $n\geq 1$, $A(\cH(n)^{O(n)})$ is a commutative algebra.\end{thm}
\begin{proof} In the case $n=1$, the Zhu algebra $A(\cH^+)$ is clearly abelian since it has two generators, one of which is central since it corresponds to the Virasoro element \cite{DNI}. For all $n\geq 1$, the circle products $j^{2l}\circ_k j^{2m}$ among the generators of $\cH(n)^{O(n)}$ are independent of $n$ for $0\leq k<2l+2m+3$. For $k = 2l+2m+3$, $j^{2l}\circ_k j^{2m}$ is a constant (which depends on $n$), and $j^{2l}\circ_k j^{2m} = 0$ for $k> 2l+2m+3$. Let $a^{2m}= \pi_{Zh}(j^{2m})$. It is clear from (\ref{zhucomm}) that the commutator $[a^{2l},a^{2m}]$ depends only on $j^{2l}\circ_k j^{2m}$ for $0\leq k\leq 2l+2$, and therefore is not affected by the value of this constant. So this commutator is the same for all $n$, and in particular must vanish since it vanishes for $n=1$. Note that this argument is independent of Conjecture \ref{mainconj}. \end{proof}

\begin{cor}All irreducible, admissible $\cH(n)^{O(n)}$-modules are highest-weight modules.\end{cor}

Since $\cV_n$ has the same generators and OPE relations as $\cH(n)^{O(n)}$, the same argument shows that the Zhu algebra of $\cV_n$ is abelian. Since $\cV_n$ is freely generated by $J^0,J^2,\dots$, it follows that $A(\cV_n)$ is the polynomial algebra $\mathbb{C}[A^0,A^2,\cdots]$, where $A^{2m} = \pi_{Zh}(J^{2m})$. Moreover, $A(\cH(n)^{O(n)}) \cong  \mathbb{C}[a^0,a^2,\dots] / I_n$, where $I_n = \pi_{Zh}(\cI_n)$, and we have a commutative diagram

\begin{equation}\label{commdiag} \begin{array}[c]{ccc}
\cV_n &\stackrel{\pi_n}{\rightarrow}& \cH(n)^{O(n)}  \\
\downarrow\scriptstyle{\pi_{Zh}}&&\downarrow\scriptstyle{\pi_{Zh}}\\
A(\cV_n) &\stackrel{A(\pi_n)}{\rightarrow}& A(\cH(n)^{O(n)})
\end{array} .\end{equation}

For $n=1$, it was shown by Dong-Nagatomo \cite{DNI} that $A(\cH^+) \cong \mathbb{C}[x,y] / I$, where $I$ is the ideal generated by the polynomials $$P=(y+x-4x^2)(70y+908x^2-515x+27),\ \ \ \ \ \ Q = (x-1)(x-\frac{1}{16})(x-\frac{9}{16})(y+x-4x^2).$$It follows that the irreducible, admissible $\cH^+$-modules are parametrized by the points on the variety $V(I)\subset \mathbb{C}^2$. Suppose now that Conjecture \ref{mainconj} holds. Then $A(\cH(n)^{O(n)})$ is generated by $\{a^0, a^2,\dots,a^{n^2+3n-2}\}$, and it follows that $$A(\cH(n)^{O(n)}) \cong \mathbb{C}[a^0,a^2,\dots, a^{n^2+3n-2}]/ I_{n},$$ where $I_{n}$ is now regarded as an ideal inside $\mathbb{C}[a^0,a^2\dots, a^{n^2+3n-2}]$. The corresponding variety $V(I_{n})\subset \mathbb{C}^{\frac{1}{2}(n^2+3n)}$ then parametrizes the irreducible, admissible modules over $\cH(n)^{O(n)}$. The problem of classifying these modules is equivalent to giving a description of the ideal $I_n$.

\section{Invariant subalgebras of $\cH(n)$ under arbitrary reductive groups}\label{secgeneral}
In this section, we study $\cH(n)^G$ for a general reductive group $G\subset O(n)$. Our approach is similar to our study of invariant subalgebras of ghost systems in \cite{LII}. By a fundamental result of Dong-Li-Mason \cite{DLM}, $\cH(n)$ has a decomposition of the form \begin{equation}\label{dlmdecomp} \cH(n) \cong \bigoplus_{\nu\in H} L(\nu)\otimes M^{\nu},\end{equation} where $H$ indexes the irreducible, finite-dimensional representations $L(\nu)$ of $O(n)$, and the $M^{\nu}$'s are inequivalent, irreducible, highest-weight $\cH(n)^{O(n)}$-modules. The modules $M^{\nu}$ appearing in (\ref{dlmdecomp}) have an integrality property; the eigenvalues of $\{j^{2m}(2m+1)|~m\geq 0\}$ on the highest-weight vectors $f_{\nu}$ are all integers. These modules therefore correspond to certain rational points on the variety $V(I_n)$, if it exists.  

Using the decomposition (\ref{dlmdecomp}), together with a classical theorem of Weyl, we show that $\cH(n)^G$ is finitely generated as a vertex algebra. This statement is analogous to Lemma 2 of \cite{LII}, and is independent of Conjecture \ref{mainconj}. We then prove some combinatorial properties of the modules $M^{\nu}$ appearing in (\ref{dlmdecomp}), which are also independent of Conjecture \ref{mainconj}. Next, we show that Conjecture \ref{mainconj} implies that each module $M^{\nu}$ appearing in (\ref{dlmdecomp}) possesses a certain finiteness property. Together with the finite generation of $\cH(n)^G$, this is enough to prove that $\cH(n)^G$ is strongly finitely generated. Since Conjecture \ref{mainconj} holds for $n=2$ and $n=3$, the strong finite generation of $\cH(2)^G$ and $H(3)^G$ is an immediate consequence.

\begin{thm} \label{ordfg} For any reductive $G\subset O(n)$, $\cH(n)^G$ is finitely generated as a vertex algebra.\end{thm}

\begin{proof} Recall that $\cH(n)\cong gr(\cH(n))$ as linear spaces, and $$gr(\cH(n)^G )\cong (gr(\cH(n))^G \cong (Sym \bigoplus_{k\geq 0} V_k)^G = R$$ as commutative algebras, where $V_k\cong \mathbb{C}^n$ as $O(n)$-modules. For all $p\geq 0$, there is an action of $GL_p$ on $\bigoplus_{k =0}^{p-1} V_k $ which commutes with the action of $G$. The natural inclusions $GL_p\hookrightarrow GL_q$ for $p<q$ sending $$M \ra  \bigg[ \begin{matrix} M & 0 \cr 0 & I_{q-p} \end{matrix} \bigg]$$ induce an action of $GL_{\infty} = \lim_{p\ra \infty} GL_p$ on $\bigoplus_{k\geq 0} V_k$. We obtain an action of $GL_{\infty}$ on $Sym \bigoplus_{k\geq 0} V_k$ by algebra automorphisms, which commutes with the action of $G$. Hence $GL_{\infty}$ acts on $R$ as well. By a basic theorem of Weyl, $R$ is generated by the set of translates under $GL_{\infty}$ of any set of generators for $(Sym \bigoplus_{k = 0} ^{n-1} V_k)^G$ \cite{W}. Since $G$ is reductive, $(Sym \bigoplus_{k = 0} ^{n-1} V_k)^G$ is finitely generated. Hence there exists a finite set of homogeneous elements $\{f_1,\dots, f_k\}\subset R$ such that $\{ \sigma f_i|~ i=1,\dots,k,~ \sigma\in GL_{\infty}\}$ generates $R$. It follows from Lemma \ref{recon} that the set of vertex operators $$\{(\sigma f_i)(z)\in \cH(n)^G|~i=1,\dots,k,~ \sigma\in GL_{\infty}\}$$ which correspond to $\sigma f_i$ under the linear isomorphism $\cH(n)^G\cong gr(\cH(n)^G) \cong R$, is a set of strong generators for $\cH(n)^G$.

In the decomposition (\ref{dlmdecomp}) of $\cH(n)$ as a bimodule over $O(n)$ and $\cH(n)^{O(n)}$, the $O(n)$-isotypic component of $\cH(n)$ of type $L(\nu)$ is isomorphic to $L(\nu)\otimes M^{\nu}$. Each $L(\nu)$ is a module over $G\subset O(n)$, and since $G$ is reductive, it has a decomposition $L(\nu) =\oplus_{\mu\in H^{\nu}} L(\nu)_{\mu}$. Here $\mu$ runs over a finite set $H^{\nu}$ of irreducible, finite-dimensional representations $L(\nu)_{\mu}$ of $G$, possibly with multiplicity. We thus obtain a refinement of (\ref{dlmdecomp}):
\begin{equation}\label{decompref} \cH(n) \cong \bigoplus_{\nu\in H} \bigoplus_{\mu\in H^{\nu}} L(\nu)_{\mu} \otimes M^{\nu}.\end{equation}
Let $f_1(z),\dots,f_k(z)\in \cH(n)^G$ be the vertex operators corresponding to the polynomials $f_1, \dots,f_k$ under the linear isomorphism $\cH(n)^G\cong gr(\cH(n)^G) \cong R$. Clearly $f_1(z),\dots, f_k(z)$ must live in a finite direct sum
\begin{equation}\label{newsummation} \bigoplus_{j=1}^r L(\nu_j)\otimes M^{\nu_j}\end{equation} of the modules appearing in (\ref{dlmdecomp}). By enlarging the collection $f_1(z),\dots,f_k(z)$ if necessary, we may assume without loss of generality that each $f_i(z)$ lives in a single representation of the form $L(\nu_j)\otimes M^{\nu_j}$. Moreover, we may assume that $f_i(z)$ lives in a trivial $G$-submodule $L(\nu_j)_{\mu_0} \otimes M^{\nu_j}$, where $\mu_0$ denotes the trivial, one-dimensional $G$-module. (In particular, $L(\nu_j)_{\mu_0}$ is one-dimensional). Since the actions of $GL_{\infty}$ and $O(n)$ on $\cH(n)$ commute, we may assume that $(\sigma f_i)(z)\in L(\nu_j)_{\mu_0}\otimes M^{\nu_j}$ for all $\sigma\in GL_{\infty}$. Since $\cH(n)^G$ is strongly generated by the set $\{ (\sigma f_i)(z)|~i=1,\dots,k,~ \sigma\in GL_{\infty}\}$, and each $M^{\nu_j}$ is an irreducible $\cH(n)^{O(n)}$-module, $\cH(n)^G$ is generated as an algebra over $\cH(n)^{O(n)}$ by $f_1(z),\dots,f_k(z)$. Finally, since $\cH(n)^{O(n)}$ is itself a finitely generated vertex algebra by Lemma \ref{ordfingen}, we conclude that $\cH(n)^G$ is finitely generated. \end{proof}

Next, we need a fact about representations of associative algebras which can be found in \cite{LII}. Let $A$ be an associative $\mathbb{C}$-algebra (not necessarily unital), and let $W$ be a linear representation of $A$, via an algebra homomorphism $\rho: A\ra End(W)$. Regarding $A$ as a Lie algebra with commutator as bracket, let $\rho_{Lie}:A\ra End(W)$ denote the map $\rho$, regarded now as a Lie algebra homomorphism. There is an induced algebra homomorphism $U(A)\ra End(W)$, where $U(A)$ denotes the universal enveloping algebra of $A$. Given elements $a,b\in A$, we denote the product in $U(A)$ by $a*b$ to distinguish it from $ab\in A$. Given a monomial $\mu = a_1* \cdots * a_r\in U(A)$, let $\tilde{\mu} = a_1\cdots a_r$ be the corresponding element of $A$. Let $U(A)_+$ denote the augmentation ideal (i. e., the ideal generated by $A$), regarded as an associative algebra with no unit. The map $U(A)_+ \ra A$ sending $\mu\mapsto \tilde{\mu}$ is then an algebra homomorphism which makes the diagram \begin{equation}\label{commutativediag} \begin{matrix} U(A)_+ & & \cr \downarrow & \searrow \cr A & \ra & End(W) \cr \end{matrix} \end{equation} commute.
Let $Sym(W)$ denote the symmetric algebra of $W$, whose $d$th graded component is denoted by $Sym^d(W)$. Clearly $\rho_{Lie}$ (but not $\rho$) can be extended to a Lie algebra homomorphism $\hat{\rho}_{Lie}: A\ra End(Sym(W))$, where $\hat{\rho}_{Lie}(a)$ acts by derivation on each $Sym^d(W)$: $$\hat{\rho}_{Lie}(a)( w_1\cdots w_d) = \sum_{i=1}^d w_1 \cdots  \hat{\rho}_{Lie}(a)(w_i)  \cdots  w_d.$$ This extends to an algebra homomorphism $U(A)\ra End(Sym(W))$ which we also denote by $\hat{\rho}_{Lie}$, but there is no commutative diagram like (\ref{commutativediag}) because the map $A\ra End(Sym(W))$ is not a map of associative algebras. In particular, the restrictions of $\hat{\rho}_{Lie}(\mu)$ and $\hat{\rho}_{Lie}(\tilde{\mu})$ to $Sym^d(W)$ are generally not the same for $d>1$. The following result appears in \cite{LII}.

\begin{lemma} \label{first} Given $\mu \in U(A)$ and $d\geq 1$, define a linear map $\Phi^d_{\mu} \in End(Sym^d(W))$ by \begin{equation}\label{mapmu} \Phi^d_{\mu}  = \hat{\rho}_{Lie}(\mu) \big|_{Sym^d(W)}.\end{equation} Let $E$ denote the subspace of $End(Sym^d(W))$ spanned by $\{\Phi^d_{\mu}|~\mu\in U(A)\}$. Note that $E$ has a filtration $$E_1\subset E_2\subset \cdots,\ \ \ \ \ \ E = \bigcup_{r\geq 1} E_r,$$ where $E_r$ is spanned by $\{\Phi^d_{\mu}|~ \mu \in U(A),~ deg(\mu) \leq r\}$. Then $E = E_d$.\end{lemma}

\begin{proof} Given a monomial $\mu = a_1* \cdots * a_r \in U(A)$ of arbitrary degree $r>d$, we need to show that $\Phi^d_{\mu}$ can be expressed as a linear combination of elements of the form $\Phi^d_{\nu}$ where $\nu \in U(A)$ and $deg (\nu) \leq d$. Fix $p\leq d$, and let $Part^r_{p}$ denote the set of partitions $\phi$ of $\{1,\dots,r\}$ into $p$ disjoint, non-empty subsets $S^{\phi}_1,\dots, S^{\phi}_p$ whose union is $\{1,\dots,r\}$. Each subset $S^{\phi}_i$ is of the form $$S^{\phi}_i = \{i_1, \dots, i_{k_i} \},\ \ \ \ \ \  i_1<\cdots < i_{k_i}.$$ For $i=1,\dots, p$, let $m_i\in U(A)$ be the corresponding monomial $m_i = a_{i_1} *\cdots *a_{i_{k_i}}$. Let $J = (j_1,\dots,j_p)$ be an (ordered) subset of $\{1,\dots,d\}$. Define a linear map $g_{\phi}\in End(Sym^d(W))$ by \begin{equation}\label{fphi} g_{\phi}(w_1 \cdots  w_d) =  \sum_J g_{\phi}^1(w_1)  \cdots g_{\phi}^d (w_d),\ \ \ \ \ \  g_{\phi}^k (w_k) =  \bigg\{ \begin{matrix} \hat{\rho}_{Lie}(m_i) (w_{j_i}) & k= j_i \cr & \cr w_k & k\neq j_i \end{matrix},\end{equation}
where the sum runs over all (ordered) $p$-element subsets $J$ as above. Note that we could replace $m_i\in U(A)$ with $\tilde{m_i}\in A$ in (\ref{fphi}), since $\hat{\rho}_{Lie}(m_i)(w_{j_i}) = \hat{\rho}_{Lie}(\tilde{m}_i)(w_{j_i})$. 

We claim that for each $\phi\in Part^r_p$, $g_{\phi} \in E_d$. We proceed by induction on $p$. The case $p=1$ is trivial because $g_{\phi} = \hat{\rho}_{Lie}(a)$ as derivations on $Sym^d(W)$, where $a = a_1\cdots a_r$. Next, assume the result for all partitions $\psi\in Part^s_{q}$, for $q<p$ and $s\leq r$. Let $m_1,\dots,m_p\in U(A)$ be the monomials corresponding to $\phi$ as above, and define $m_{\phi} = \tilde{m}_1*\cdots * \tilde{m}_p \in U(A)$. By definition, $\Phi^d_{m_{\phi}} \in E_p\subset E_d$, and the leading term of $\Phi^d_{m_{\phi}}$ is $g_{\phi}$. The lower order terms are of the form $g_{\psi}$, where $\psi\in Part^p_q$ is a partition of $\{1,\dots, p\}$ into $q$ subsets, which each corresponds to a monomial in the variables $\tilde{m}_1,\dots,\tilde{m}_p$.
By induction, each of these terms lies in $E_q$, and since $g_{\phi} \equiv \Phi^d_{m_{\phi}}$ modulo $E_q$, the claim is proved.

Finally, using the derivation property of $A$ acting on $Sym^d(W)$, one checks easily that \begin{equation}\label{triplesum} \Phi^d_{\mu} = \sum_{p=1}^d \sum_{\phi\in Part^r_p} g_{\phi}.\end{equation} Since each $g_{\phi}$ lies in $E_d$ by the above claim, this completes the proof of the lemma. \end{proof}

\begin{cor} \label{firstcor} Let $f\in Sym^d(W)$, and let $M\subset Sym^d(W)$ be the cyclic $U(A)$-module generated by $f$. Then $\{\hat{\rho}_{Lie}(\mu)(f)|~ \mu\in U(A),~ deg(\mu)\leq d\}$ spans $M$.\end{cor}

Let $\cL$ denote the Lie algebra generated by the Fourier modes $\{j^{2m}(k)|~k\in\mathbb{Z},~m\geq 0\}$ of the generators of $\cH(n)^{O(n)}$, and let $\cP\subset \cL$ be the subalgebra generated by the annihilation modes $\{j^{2m}(k)|~k\geq 0\}$. Note that $\cP$ has a decomposition $$\cP = \cP_- \oplus \cP_0\oplus \cP_+,$$ where $\cP_-$, $\cP_0$, and $\cP_+$ are the Lie algebras spanned by $\{j^{2m}(k)|~0\leq k< 2m+1\}$, $\{j^{2m}(2m+1)\}$, and  $\{j^{2m}(k)|~k>2m+1\}$, respectively. Clearly $\cP$ preserves the filtration on $\cH(n)$, so each element of $\cP$ acts by a derivation of degree zero on $gr(\cH(n))$. 

Let $\cM$ be an irreducible, highest-weight $\cH(n)^{O(n)}$-submodule of $\cH(n)$ with generator $f(z)$, and let $\cM'$ denote the $\cP$-submodule of $\cM$ generated by $f(z)$. Since $f(z)$ has minimal weight among elements of $\cM$ and $\cP_+$ lowers weight, $f(z)$ is annihilated by $\cP_+$. Moreover, $\cP_0$ acts diagonalizably on $f(z)$, so $f(z)$ generates a one-dimensional $\cP_0\oplus \cP_+$-module. By the Poincare-Birkhoff-Witt theorem, $\cM'$ is a quotient of $$U(\cP)\otimes_{U(\cP_0\oplus \cP_+)} \mathbb{C} f(z),$$ and in particular is a cyclic $\cP_-$-module with generator $f(z)$. Suppose that $f(z)$ has degree $d$, that is, $f(z)\in \cH(n)_{(d)} \setminus \cH(n)_{(d-1)}$. Since each element of $\cP$ preserves the filtration on $\cH(n)$, and $\cM$ is irreducible, it is easy to see that the nonzero elements of $\cM'$ lie in $\cH(n)_{(d)} \setminus \cH(n)_{(d-1)}$. Therefore, the projection $\cH(n)_{(d)}\ra \cH(n)_{(d)}/\cH(n)_{(d-1)} \subset gr(\cH(n))$ restricts to an isomorphism of $\cP$-modules
\begin{equation}\label{isopmod} \cM'\cong gr(\cM')\subset gr(\cS(V)).\end{equation} By Lemma \ref{firstcor}, we conclude that $\cM'$ is spanned by elements of the form $$\{j^{2l_1}(k_1)\cdots j^{2l_r}(k_r) f(z) |~ j^{2l_i}(k_i)\in \cP_-,~ r\leq d\}.$$

Next, we need a basic fact from linear algebra. Let $A = (A_{i,j})$, $i,j = 1,\dots,n$ be an $n\times n$-matrix, whose entries $A_{i,j}$ are all positive real numbers. We call $A$ {\it totally increasing} if $A_{i,j} \leq A_{i,j+1}$ and $A_{i,j}\leq A_{i+1,j}$ for all $i,j$, and for each $i,j = 1,\dots,n-1$, the $2\times 2$ matrix $$\bigg[ \begin{matrix} A_{i,j} & A_{i,j+1} \cr  A_{i+1,j} & A_{i+1,j+1}\end{matrix} \bigg]$$ has positive determinant.

\begin{lemma} \label{totinc} Any totally increasing matrix is nonsingular.\end{lemma}
\begin{proof} Let $A$ be such a matrix. First, for $i=1,\dots,n$ we rescale the $i$th row by a factor of $1/A_{i,1}$. Next, for $j=2,\dots,n$, we rescale the $j$th column by a factor of $A_{1,1}/A_{1,j}$ to obtain $$\left[ \begin{matrix}1 & 1 & 1 & \cdot & \cdot & \cdot &  1 \cr   1  & \frac{A_{1,1} A_{2,2}}{A_{2,1}A_{1,2}} & \frac{A_{1,1}A_{2,3}}{A_{2,1}A_{1,3}} & \cdot & \cdot &   \cdot & \frac{A_{1,1}A_{2,n}}{A_{2,1}A_{1,n}}  \cr  \vdots & \vdots & \vdots &  & & & \vdots  \cr  1  & \frac{A_{1,1}A_{n,2}}{A_{n,1}A_{1,2}}   & \frac{A_{1,1} A_{n,3}}{A_{n,1}A_{1,3}} & \cdot & \cdot  &  \cdot &  \frac{A_{1,1} A_{n,n}}{A_{n,1} A_{1,n}} \end{matrix}\right],$$ which is easily seen to be totally increasing. Finally, we subtract the first row from the $i$th row, for $i=2,\cdots,n$, obtaining $$\left[ \begin{matrix}1 & 1 & 1 & \cdot & \cdot & \cdot &  1 \cr   0  & \frac{A_{1,1} A_{2,2}}{A_{2,1}A_{1,2}} -1 & \frac{A_{1,1}A_{2,3}}{A_{2,1}A_{1,3}} -1 & \cdot & \cdot &   \cdot & \frac{A_{1,1}A_{2,n}}{A_{2,1}A_{1,n}} -1 \cr  \vdots & \vdots & \vdots &  & & & \vdots  \cr  0  & \frac{A_{1,1}A_{n,2}}{A_{n,1}A_{1,2}} -1  & \frac{A_{1,1} A_{n,3}}{A_{n,1}A_{1,3}} -1 & \cdot & \cdot  &  \cdot &  \frac{A_{1,1} A_{n,n}}{A_{n,1} A_{1,n}} -1 \end{matrix} \right].$$ It is easy to check that the $(n-1)\times (n-1)$ matrix 
$$\left[ \begin{matrix} \frac{A_{1,1} A_{2,2}}{A_{2,1}A_{1,2}} -1 & \frac{A_{1,1}A_{2,3}}{A_{2,1}A_{1,3}} -1 & \cdot & \cdot &   \cdot & \frac{A_{1,1}A_{2,n}}{A_{2,1}A_{1,n}} -1 \cr  \vdots & \vdots &  &  & & \vdots  \cr  \frac{A_{1,1}A_{n,2}}{A_{n,1}A_{1,2}} -1  & \frac{A_{1,1} A_{n,3}}{A_{n,1}A_{1,3}} -1 & \cdot & \cdot  &  \cdot &  \frac{A_{1,1} A_{n,n}}{A_{n,1} A_{1,n}} -1 \end{matrix} \right]$$ is totally increasing, so the claim follows by induction on $n$. \end{proof}

This lemma allows us to prove a certain useful property of $\cH(n)$ as a module over $\cP$. For simplicity of notation, we take $n=1$, but the result we are going to prove holds for any $n$. In this case, $\cH = \cH(1)$ is generated by $\alpha(z)$. Recall from (\ref{structureofgrs}) that $\alpha_j$ denotes the image of $\partial^j\alpha$ in $gr(\cH)$. Let $W\subset gr(\cH)$ be the vector space with basis $\{\alpha_j|~ j\geq 0\}$, and for each $m\geq 0$, let $W_m$ be the subspace with basis $\{\alpha_j|~ 0\leq j\leq m\}$. Let $\phi:W\ra W$ be a linear map of weight $w\geq 1$, such that \begin{equation}\label{arbmap} \phi(\alpha_j) = c_j \alpha_{j+w},\end{equation} for constants $c_j \in \mathbb{C}$. For example, the restriction $j^{2k}(2k-w+1)\big|_{W}$ of any $j^{2k}(2k-w+1)\in \cP$, is such a map.

\begin{lemma} \label{third} Fix $w\geq 1$ and $m\geq 0$, and let $\phi$ be a linear map satisfying (\ref{arbmap}). Then the restriction $\phi \big|_{W_m}$ can be expressed uniquely as a linear combination of the operators $j^{2k}(2k-w+1)\big|_{W_m}$ for $0\leq 2k+1-w \leq 2m+1$.\end{lemma}

\begin{proof} Suppose first that $w$ is odd, and let $k_j = j+ \frac{1}{2}(w-1)$, for $j=0,\dots,m$. In this notation, we need to show that $\phi \big|_{W_m}$ can be expressed uniquely as a linear combination of the operators $j^{2k_j}(2j)\big|_{W_m}$ for $j=0,\dots,m$. Using (\ref{bcalc}), we calculate 
 \begin{equation}\label{actionofp} j^{2k_j}(2j)(\alpha_i) = \lambda_{0,2k_j,i,2j} (\alpha_{i+w}) =  \bigg(\frac{(2k_j+i+1)!}{(2k_j+i+1-2j)!} +  \frac{(i+1)!}{(i+1-2j)!} \bigg)\alpha_{i+w}.\end{equation}
 
Let $M^w$ be the $(m+1)\times(m+1)$ matrix with entries $M^w_{i,j} =  \lambda_{0,2k_j,i,2j}$, for $i,j = 0,\dots,m$. Let ${\bf c}$ be the column vector in $\mathbb{C}^{m+1}$ whose transpose is given by $(c_0,\dots,c_m)$. Given an arbitrary linear combination $$\psi = t_0 j^{2k_0}(0) + t_1 j^{2k_1}(2) + \cdots + t_{m} j^{2k_m}(2m)$$ of the operators $j^{2k_j}(2j)$ for $0\leq j\leq m$, let ${\bf t}$ be the column vector whose transpose is $(t_0,\dots, t_{m})$. Note that $\phi \big|_{W_m} = \psi \big|_{W_m}$ precisely when $M^w  {\bf t} =  {\bf c}$, so in order to prove the claim, it suffices to show that $M^w$ is invertible. It is easy to check using (\ref{actionofp}) that $M^w$ is totally increasing, so this is immediate from Lemma \ref{totinc}. Finally, if $w$ is even, the same argument shows that for $k_j = j+\frac{w}{2}$, $j=0,\dots,m$, $\phi$ can be expressed uniquely as a linear combination of the operators $j^{2k_j}(2j+1)$ for $j=0,\dots,m$. \end{proof}

Since (\ref{bcalc}) holds for any $n\geq 1$, it follows that the statement of Lemma \ref{third} holds for any $n$. More precisely, let $W\subset gr(\cH(n))$ be the vector space with basis $\{\alpha^i_j|~i=1,\dots,n,~j\geq 0\}$, and let $W_m\subset W$ be the subspace with basis $\{\alpha^i_j|~i=1,\dots, n,~0\leq j\leq m\}$. Let $\phi:W\ra W$ be a linear map of weight $w\geq1$ taking \begin{equation}\label{actiongencase} \alpha^i_j\mapsto c_j \alpha^i_{j+w},\ \ \ \ \ \ i=1,\dots,n,\end{equation} where the constants $c_j$ are independent of $i$. For example, each $\phi = j^{2k}(2k-w+1)\big|_W$ satisfies (\ref{actiongencase}). Then $\phi\big|_{W_m}$ can be expressed uniquely as a linear combination of $j^{2k}(2k-w+1)\big|_{W_m}$ for $0\leq 2k+1-w \leq 2m+1$. 

The next result is analogous to Lemma 7 of \cite{LII}, and the proof is almost identical. 
\begin{lemma} \label{fourth} Let $\cM$ be an irreducible, highest-weight $\cH(n)^{O(n)}$-submodule of $\cH(n)$ with highest-weight vector $f(z)$ of degree $d$. Let $\cM'$ be the corresponding $\cP$-module generated by $f(z)$, and let $f$ be the image of $f(z)$ in $gr(\cH(n))$, which generates $M = gr(\cM')$ as a $\cP$-module. Fix $m$ so that $f\in Sym^d(W_m)$. Then $\cM'$ is spanned by $$\{j^{2l_1}(k_1) \cdots j^{2l_r}(k_r) f(z)|~j^{2l_i}(k_i)\in \cP_-,\ \  r\leq d,\ \ 0\leq k_i \leq 2m+1\}.$$\end{lemma}

\begin{proof} We may work with $M = gr(\cM')$ rather than $\cM'$, and for notational convenience, we do not distinguish between elements of $U(\cP_-)$ and their images in $End(Sym^d(W))$. As in the proof of Lemma \ref{first}, let $E$ denote the subspace of $End(Sym^d (W))$ spanned by $U(\cP_-)$, and let $E_r$ be the subspace spanned by elements of $U(\cP_-)$ of degree at most $r$. Let $\tilde{E}_r$ be the subspace of $E_r$ spanned by elements of $U(\cP_-)$ which only depend on $j^{2l}(k)$ for $k\leq 2m+1$. 

It is not true that $E_d = \tilde{E}_d$ as subspaces of $End(Sym^d(W))$, but it suffices to show that these spaces of endomorphisms coincide when restricted to $Sym^d(W_m)$. Since $E = E_d$, and hence is spanned by monomials $\mu = a_1*\cdots * a_r\in U(\cP_-)$ of degree $r\leq d$, we have 
\begin{equation}\label{triplesumii} \mu = \sum_{p=1}^r  \sum_{\phi\in Part^r_p} g_{\phi},\end{equation} where each partition $\phi\in Part^r_p$ corresponds to a set of monomials $m_1,\dots,m_p$, and $g_{\phi}$ is given by (\ref{fphi}). For $p=r$, there is only one partition $\phi_0$ of $\{1,\dots,r\}$ into disjoint, non-empty subsets, and $g_{\phi_0}$ is defined on monomials $w_1\cdots w_d\in Sym^d(W)$ by \begin{equation}\label{fphii} g_{\phi_0}(w_1\cdots w_d) =  \sum_J g_{\phi_0}^1(w_1)  \cdots  g_{\phi_0}^d (w_d),\ \ \ \ \ \  g_{\phi_0}^k (w_k) =  \bigg\{ \begin{matrix} a_i (w_{j_i}) & k= j_i \cr & \cr w_k & k\neq j_i\end{matrix},\end{equation} where the sum runs over all (ordered) $r$-element subsets $J \subset \{1,\dots,d\}$. By Lemma \ref{third}, the restriction of $a_i$ to $W_m$ coincides with a linear combination $S_i$ of the elements $j^{2l}(k)\big|_{W_m}$ for $k\leq 2m+1$. Replace each of the factors $a_i (w_{j_i})$ appearing in (\ref{fphii}) with $S_i (w_{j_i})$, and let $Q = \prod_{i=1}^r S_i$, which lies in $U(\cP_-)$, and depends only on $j^{2l}(k)$ for $k\leq 2m+1$. Clearly the restriction of $Q$ to $Sym^d(W_m)$ agrees with the restriction of $\mu$ to $Sym^d(W_m)$, modulo terms lying in $E_{r-1}$. The lemma then follows by induction on $r$. \end{proof}

As in \cite{LII}, we may order the elements $j^{2l}(k)\in \cP_-$ as follows: $j^{2l_1}(k_1) > j^{2l_2}(k_2)$ if $l_1>l_2$, or $l_1=l_2$ and $k_1<k_2$. Then Lemma \ref{fourth} can be strengthened as follows: $\cM'$ is spanned by elements of the form $j^{2l_1}(k_1)\cdots j^{2l_r}(k_r) f(z)$ with \begin{equation}\label{shapealpha} j^{2l_i}(k_i)\in \cP_-,\ \ \ \  r\leq d,\ \ \ \ 0\leq k_i\leq 2m+1,\ \ \ \  j^{2l_1}(k_1)\geq \cdots \geq j^{2l_r}(k_r).\end{equation}

Up to this point, everything we have proven in this section is independent of Conjecture \ref{mainconj}. Then next lemma is where this assumption will enter. We use the notation $\cH(n)^{O(n)}[k]$, $\cM[k]$, and $\cM'[k]$ to denote the homogeneous components of these spaces of conformal weight $k$. As in \cite{LII}, we define the {\it Wick ideal} $\cM_{Wick}\subset \cM$ to be the subspace spanned by elements of the form $$:a(z) b(z):,\ \ \ \ \ a(z)\in \bigoplus_{k>0} \cH(n)^{O(n)}[k],\ \ \ \ \ b(z)\in \cM.$$ Despite the choice of terminology, $\cM_{Wick}$ is not a vertex algebra ideal. It is properly contained in the space $C_1(\cM)$ defined in \cite{LiIII}, and in particular it does not contain all elements of the form $L\circ_0 b = \partial b$ for $b\in\cM$.

\begin{lemma} \label{fifth} Let $\cM$ be an irreducible, highest-weight $\cH(n)^{O(n)}$-submodule of $\cH(n)$ with highest-weight vector $f(z)$. If Conjecture \ref{mainconj} holds, any homogeneous element of $\cM$ of sufficiently high weight lies in the Wick ideal. In particular, $\cM / \cM_{Wick}$ is finite-dimensional.\end{lemma}

\begin{proof} It suffices to show that $\cM'[k]$ lies in the Wick ideal for $k>>0$, where $\cM'$ is the $\cP$-module generated by $f(z)$. As usual, let $d$ be the degree of $f(z)$, and fix $m$ so that $f \in Sym^d(W_m)$. Recall that $\cM'$ is spanned by elements of the form $j^{2l_1}(k_1)\cdots j^{2l_r}(k_r) f(z)$ satisfying (\ref{shapealpha}). Fix an element $\alpha(z)$ of this form of weight $K>>0$. Since each operator $j^{2l_i}(k_i)$ has weight $2l_i+1-k_i$, $k_i\leq 2m+1$, and $K>>0$, we may assume that $l_1>>\frac{1}{2}(n^2+3n)$. Then the decoupling relation (\ref{maindecoup}) allows us to express $j^{2l_1}(z)$ as a normally ordered polynomial $Q_{l_1}(z)$ in the generators \begin{equation}\label{genera} \partial^t j^{2l}(z),\ \ \ \ \ 0\leq l\leq \frac{1}{2}(n^2+3n-2),\ \ \ \ \ t\geq 0.\end{equation} We claim that for any weight-homogeneous, normally ordered polynomial $Q(z)$ in the generators (\ref{genera}) of sufficiently high weight, any element $c(z)\in \cM$, and any $k$ satisfying $0\leq k\leq 2m+1$, $Q(z)\circ_k c(z)$ lies in $\cM_{Wick}$. Specializing this to the case $Q(z) = Q_{l_1}(z)$, $c(z) = j^{2l_2}(k_2)\cdots j^{2l_r}(k_r) f(z)$, and $k=k_1$, proves the lemma.
 
We may assume without loss of generality that $Q(z)=:a(z)b(z):$ where $a(z) = \partial^t j^{2l}(z)$ for some $0\leq l\leq \frac{1}{2}(n^2 +3n-2)$. Then using (\ref{vaidiv}), and suppressing the formal variable $z$, we have \begin{equation}\label{appvaid} Q\circ_k c = \big(:ab:\big) \circ_{k} c = \sum_{r\geq0}{1\over r!}:(\partial^r a)(b\circ_{k+r}c):
+\sum_{r\geq 0}b\circ_{k-r-1}(a\circ_r c) .\end{equation}  Suppose first that $b = \lambda 1$ for some constant $\lambda$. Then $Q= \lambda \partial^t j^{2l}$, and since $wt(Q) >>0$, we have $t>>0$. Hence $Q\circ_k= \lambda(\partial^t j^{2l})\circ_k = 0$ as an operator (since this operator vanishes whenever $t>k$). So we may assume without loss of generality that $b$ is not a constant. 

We proceed by induction on $k$. For $k=0$, each term appearing in (\ref{appvaid}) lies in $\cM_{Wick}$, so there is nothing to prove. For $k>0$, the only terms appearing in (\ref{appvaid}) that need not lie in $\cM_{Wick}$ a priori, are those of the form $\sum_{r=0}^{k-1} b\circ_{k-r-1}(a\circ_r c)$. However, each of these terms is weight-homogeneous, and the weight of $a\circ_r c = \partial^t j^{2l} \circ_r c$ is bounded above by $wt(c) + n^2+3n+1$, since $\partial^t j^{2l} \circ_r c=0$ for $t>r$. So we may still assume that $wt(b)>>0$. By our inductive assumption, all these terms then lie in $\cM_{Wick}$. \end{proof}

\begin{cor} \label{sixth} Let $\cM$ be an irreducible, highest-weight $\cH(n)^{O(n)}$-submodule of $\cH(n)$. Given a subset $S\subset \cM$, let $\cM_S\subset \cM$ denote the subspace spanned by elements of the form $$:\omega_1(z)\cdots \omega_t(z) \alpha(z):,\ \ \ \ \ \omega_j(z)\in \cH(n)^{O(n)},\ \ \ \ \ \alpha(z)\in S.$$ If Conjecture \ref{mainconj} holds, there exists a finite set $S\subset \cM$ such that $\cM = \cM_S$.\end{cor}

Now we are ready to prove our main result.

\begin{thm} \label{sfg} Suppose that Conjecture \ref{mainconj} holds. Then for any reductive group $G$ of automorphisms of $\cH(n)$ preserving the conformal structure (\ref{virasoro}), $\cH(n)^G$ is strongly finitely generated.\end{thm}

\begin{proof} By Theorem \ref{ordfg}, we can find vertex operators $f_1(z),\dots, f_k(z)$ such that the corresponding polynomials $f_1,\dots, f_k\in gr(\cH(n))^G$, together with all $GL_{\infty}$ translates of $f_1,\dots, f_k$, generate the invariant ring $gr(\cH(n))^G$. As in the proof of Lemma \ref{ordfg}, we may assume that each $f_i(z)$ lies in an irreducible, highest-weight $\cH(n)^{O(n)}$-module $\cM_i$ of the form $L(\nu)_{\mu_0}\otimes M^{\nu}$, where $L(\nu)_{\mu_0}$ is a trivial, one-dimensional $G$-module. Furthermore, we may assume without loss of generality that $f_1(z),\dots, f_k(z)$ are highest-weight vectors for the action of $\cH(n)^{O(n)}$. Otherwise, we can replace these with the highest-weight vectors in the corresponding modules. 

For each $\cM_i$, choose a finite set $S_i \subset \cM_i$ such that $\cM_i = (\cM_i)_{S_i}$, using Corollary \ref{sixth}. Define $$ S=\{j^0(z),j^2(z),\dots, j^{n^2+3n-2}(z) \} \cup \big(\bigcup_{i=1}^k S_i \big).$$ Since $\{j^0(z),j^2(z),\dots, j^{n^2+3n-2}(z)\}$ strongly generates $\cH(n)^{O(n)}$ (assuming Conjecture \ref{mainconj}), and the set $\bigcup_{i=1}^k \cM_i$ strongly generates $\cH(n)^G$, it is immediate that $S$ is a strong, finite generating set for $\cH(n)^G$. \end{proof}

We include one more secondary result in this section, which is independent of  Conjecture \ref{mainconj}. Recall that $\cH(n)^{O(n)}$ is the quotient of $\cV_n$ by the ideal $\cI_n$, which is generated by the set $\{D_{I,J}\}$, where $I,J$ satisfy (\ref{ijineq}). We will show that $\cI_n$ is finitely generated as a vertex algebra ideal. Define $$U_n = (\cV_{n})_{(2n+2)}\cap \cI_{n},$$ which is just the vector space spanned by $\{D_{I,J}\}$, where $I,J$ satisfy (\ref{ijineq}). We have $D_{I,J} = D_{J,I}$, but there are no other linear relations among these elements. It is easy to see that $U_n$ is a module over the Lie algebra $\cP$ generated by $\{j^{2m}(k) = j^{2m}\circ_k |~k,m\geq 0\}$, since the action $\cP$ preserves both the filtration degree and the ideal $\cI_n$. Note that $\cP$ has an alternative generating set $\{\Omega_{a,b}\circ_{a+b+1-w}|~ 0\leq a\leq b,~ a+b+1-w\geq 0\}$, where $\Omega_{a,b}\circ_{a+b+1-w}$ is homogeneous of weight $w$. Recall that $gr(\cV_n)$ is the polynomial algebra generated by $\alpha^i_k$ for $i=1,\dots,n$ and $k\geq 0$. The action of $\cP$ by derivations of degree zero on $gr(\cV_{n})$ coming from the vertex Poisson algebra structure is independent of $n$, and is specified by (\ref{bcalc}). Using this formula, it is easy to see that the action of $\cP$ on $U_n$ is by \lq\lq weighted derivation" in the following sense. Fix $I = (i_0,\dots,i_n)$ and $J = (j_0,\dots,j_n)$, and let $D_{I,J}$ be the corresponding element of $U_n$. Given $p = \Omega_{a,b}\circ_{a+b+1-w}\in \cP$, we have \begin{equation}\label{paraction} p(D_{I,J}) = \sum_{r=0}^n c_r D_{I^r,J} + \sum_{r=0}^n d_r D_{I,J^r},\end{equation} for lists $I^r = (i_0,\dots, i_{r-1}, i_r + w,i_{r+1},\dots, i_n)$ and $J^r = (j_0,\dots, j_{r-1}, j_r+ w,j_{r+1},\dots, j_n)$, and constants $c_r,d_r$. If $i_r + w$ appears elsewhere on the list $I^r$, $c_r = 0$, and if $j_r + w$ appears elsewhere on the list $J^r$, $d_r = 0$. Otherwise, \begin{equation}\label{actioni} c_r = \pm \lambda_{a,b,i_r,t},\ \ \ \ \ \ d_r = \pm \lambda_{a,b,j_r,t},\end{equation} where $t= a+b+1-w$, and the signs $\pm$ are the signs of the permutations transforming $I^r$ and $J^r$ into lists in increasing order, as in (\ref{ijineq}). 

\begin{thm} $\cI_n$ is generated as a vertex algebra ideal by the set of elements $D_{I,J} \in U_n$ for which $|I| + |J| \leq 2n^2+3n$.\end{thm}

\begin{proof} Let $\cI'_n$ denote the ideal in $\cV_n$ generated by $\{D_{I,J}|~ |I| + |J| \leq 2n^2+3n\}$. Since $U_n$ generates $\cI_n$, $\cI'_n$ is properly contained in $\cI_n$ if and only if there exists some $D_{I,J}\in \cI_n \setminus \cI'_n$. Suppose that $\cI_n \setminus \cI'_n$ is nonempty, and let $D_{I,J}$ be an element of this form of minimal weight $d$, lying in $\cI_n \setminus \cI'_n$.

Note that elements $D_{I,J}$ satisfying $|I|+|J| \leq 2n^2+3n$ have weight at most $2n^2+5n+2$. This choice guarantees that all elements $D_{I,J}$ for which the union of $I$ and $J$ is $\{0,1,\dots,n,n+1,n+2,\dots,2n+1\}$, must lie in $\cI'_n$. We say that $D_{I,J}$ has a {\it hole} at some integer $k\geq 0$ if $k$ does not appear in either $I$ or $J$. Since $d = wt(D_{I,J}) >2n^2+5n+2$, it follows that $D_{I,J}$ has a hole for some $k$ satisfying $0\leq k\leq 2n+1$, and also that there is some $l>2n+1$ such $D_{I,J}$ does {\it not} have a hole at $l$. Without loss of generality, we may assume that $l$ appears in $I$. Let $w = l-k$, and fix an integer $m$ greater than all entries of both $I$ and $J$.

By Lemma \ref{third}, we can choose an element $p\in \cP$ which is a linear combination of the operators $\Omega_{a,b} \circ_{a+b+1-w}$, for which $p(\alpha^i_k) = \alpha^i_l$ and $p(\alpha^i_r) = 0$ for all $i=1,\dots, n$ and all $r\neq k$ satisfying $0\leq r\leq m$. Let $I'$ be the list obtained from $I$ by replacing $l$ by $k$, and let $D_{I',J}$ be the corresponding element of $U_n$. Since $D_{I',J}$ has weight $d-w$, it lies in $\cI'_n$ by inductive assumption, and we clearly have $p(D_{I',J}) = D_{I,J}$. This shows that $D_{I,J}$ lies in $\cI'_n$ as well. \end{proof}

\section{Appendix}

In this Appendix, we present the results of computer calculations that prove Conjecture \ref{mainconj} in the cases $n=2$ and $n=3$. These calculations were done using Kris Thielemann's OPE package for Mathematica \cite{T}. For $n=2$, define
$$D^6_0 =:\Omega_{0,0} \Omega_{1,1} \Omega_{2,2}:  - :\Omega_{0,2} \Omega_{0,2} \Omega_{1,1}: + 2 :\Omega_{0,1} \Omega_{0,2} \Omega_{1,2}: - :\Omega_{0,0} \Omega_{1,2}\Omega_{1,2}: - :\Omega_{0,1} \Omega_{0,1} \Omega_{2,2}:,$$ 

$$D^4_0 = - 1/30 :\Omega_{0,0} \Omega_{1,7}: + 3/4 :\Omega_{0,0} \Omega_{2,6}: - 5/6: \Omega_{0,1} \Omega_{1,6}: -  13/15 :\Omega_{0,1} \Omega_{2,5}: $$ $$+ 61/30 : \Omega_{0,2} \Omega_{1,5}:  -  1/6 : \Omega_{0,2} \Omega_{2,4}:  + 5/12 :\Omega_{1,1} \Omega_{0,6}: - 1/4 :\Omega_{1,1} \Omega_{2,4}: $$ $$ -  1/10  :\Omega_{1,2} \Omega_{0,5}: + :\Omega_{1,2} \Omega_{1,4}:  + :\Omega_{1,2} \Omega_{2,3}: + 1/12 :\Omega_{0,4} \Omega_{2,2}: - 7/6 :\Omega_{1,3} \Omega_{2,2}: ,$$ 

$$ D^2_0 =   \frac{149}{600} \Omega_{0,10} +\partial^2 \bigg(- 
 \frac{1697}{10800} \Omega_{0,8} + \frac{1637}{18900} \Omega_{1,7} - \frac{33241}{37800} \Omega_{2,6} + 
 \frac{16117}{9450} \Omega_{3,5}- \frac{4223}{3780} \Omega_{4,4}\bigg) .$$
 
A computer calculation shows that the corresponding normally ordered polynomial in the variables $\omega_{a,b}$ is identically zero, so $D^6_0 + D^4_0 + D^2_0$ lies in $Ker(\pi_2)$, and hence must coincide with $D_0$. It follows that $R_0 =  \frac{149}{600} J^{10}$, and in particular is nonzero. 

Similarly, in the case $n=3$, define
$$D^8_0 = :\Omega_{0,3} \Omega_{0,3} \Omega_{1,2} \Omega_{1,2}: - 2 :\Omega_{0,2}\Omega_{0,3} \Omega_{1,2} \Omega_{1,3} : + 
 :\Omega_{0,2}\Omega_{0,2} \Omega_{1,3} \Omega_{1,3}: $$ $$- :\Omega_{0,3} \Omega_{0,3} \Omega_{1,1} \Omega_{2,2}:  + 
 2 :\Omega_{0,1} \Omega_{0,3} \Omega_{1,3} \Omega_{2,2}:  - :\Omega_{0,0}\Omega_{1,3} \Omega_{1,3} \Omega_{2,2}: $$ $$+ 
 2 :\Omega_{0,2}\Omega_{0,3}\Omega_{1,1}\Omega_{2,3}: - 2 :\Omega_{0,1}\Omega_{0,3}\Omega_{1,2}\Omega_{2,3}: - 
 2 :\Omega_{0,1}\Omega_{0,2}\Omega_{1,3}\Omega_{2,3}: $$ $$+ 2 :\Omega_{0,0}\Omega_{1,2}\Omega_{1,3}\Omega_{2,3}: + 
 :\Omega_{0,1}\Omega_{0,1}\Omega_{2,3}\Omega_{2,3}: - :\Omega_{0,0}\Omega_{1,1}\Omega_{2,3}\Omega_{2,3}: $$ $$- 
 :\Omega_{0,2}\Omega_{0,2}\Omega_{1,1}\Omega_{3,3}: + 2 :\Omega_{0,1}\Omega_{0,2}\Omega_{1,2}\Omega_{3,3}: - 
 :\Omega_{0,0}\Omega_{1,2}\Omega_{1,2}\Omega_{3,3}:$$ $$ - :\Omega_{0,1}\Omega_{0,1}\Omega_{2,2}\Omega_{3,3}: + 
 :\Omega_{0,0}\Omega_{1,1}\Omega_{2,2}\Omega_{3,3}:,$$

$$D^6_0 = 1/56 :\Omega_{2,10}\Omega_{1,1}\Omega_{0,0}: - 23/42 :\Omega_{3,9}\Omega_{1,1}\Omega_{0,0}: - 
 1/56 :\Omega_{2,10}\Omega_{0,1}\Omega_{0,1}: + 23/42 :\Omega_{3,9}\Omega_{0,1}\Omega_{0,1}: $$ $$+ 
 7/12 :\Omega_{2,9}\Omega_{1,2}\Omega_{0,0}: + 62/105 :\Omega_{3,8}\Omega_{1,2}\Omega_{0,0}:- 
 7/12 :\Omega_{2,9}\Omega_{0,2}\Omega_{0,1}:- 62/105 :\Omega_{3,8}\Omega_{0,2}\Omega_{0,1}: $$ $$- 
 139/105 :\Omega_{2,8}\Omega_{1,3}\Omega_{0,0}:+ 1/15 :\Omega_{37}\Omega_{13}\Omega_{00}: - 
 7/24 :\Omega_{1,9}\Omega_{2,2}\Omega_{0,0}: + 1/4 :\Omega_{3,7}\Omega_{2,2}\Omega_{0,0}: $$ $$+ 
 139/105 :\Omega_{2,8}\Omega_{0,3}\Omega_{0,1}: - 1/15 :\Omega_{3,7}\Omega_{0,3}\Omega_{0,1}: + 
 3/20 :\Omega_{2,8}\Omega_{1,2}\Omega_{0,1}: - 81/70 :\Omega_{3,7}\Omega_{1,2}\Omega_{0,1}: $$ $$+ 
 7/24 :\Omega_{1,9}\Omega_{0,2}\Omega_{0,2}: - 1/4 :\Omega_{3,7}\Omega_{0,2}\Omega_{0,2}: - 
 3/20 :\Omega_{2,8}\Omega_{0,2}\Omega_{1,1}: + 81/70 :\Omega_{3,7}\Omega_{0,2}\Omega_{1,1}: $$ $$+ 
 1/21 :\Omega_{1,8}\Omega_{2,3}\Omega_{0,0}: - 2/3 :\Omega_{2,7}\Omega_{2,3}\Omega_{0,0}: - 
 7/10 :\Omega_{3,6}\Omega_{2,3}\Omega_{0,0}: + 9/35 :\Omega_{2,7}\Omega_{1,3}\Omega_{0,1}: $$ $$+ 
 5/6 :\Omega_{3,6}\Omega_{1,3}\Omega_{0,1}: - 3/20 :\Omega_{1,8}\Omega_{2,2}\Omega_{0,1}: - 
 1/5 :\Omega_{3,6}\Omega_{2,2}\Omega_{0,1}: - 1/21 :\Omega_{1,8}\Omega_{0,3}\Omega_{0,2}: $$ $$+ 
 2/3 :\Omega_{2,7}\Omega_{0,3}\Omega_{0,2}: + 7/10 :\Omega_{3,6}\Omega_{0,3}\Omega_{0,2}:+ 
 3/20 :\Omega_{1,8}\Omega_{1,2}\Omega_{0,2}: + 1/5 :\Omega_{3,6}\Omega_{1,2}\Omega_{0,2}: $$ $$- 
 9/35 :\Omega_{2,7}\Omega_{0,3}\Omega_{1,1}: - 5/6 :\Omega_{3,6}\Omega_{0,3}\Omega_{1,1}: - 
 1/30 :\Omega_{1,7}\Omega_{3,3}\Omega_{0,0}: + 3/4 :\Omega_{2,6}\Omega_{3,3}\Omega_{0,0}: $$ $$+ 
 9/10 :\Omega_{1,7}\Omega_{2,3}\Omega_{0,1}: + 1/5 :\Omega_{2,6}\Omega_{2,3}\Omega_{0,1}: + 
 13/15 :\Omega_{3,5}\Omega_{2,3}\Omega_{0,1}:- 81/70 :\Omega_{1,7}\Omega_{1,3}\Omega_{0,2}: $$ $$+ 
 2/5 :\Omega_{2,6}\Omega_{1,3}\Omega_{0,2}: - 61/30 :\Omega_{3,5}\Omega_{1,3}\Omega_{0,2}:+ 
 3/40 :\Omega_{0,8}\Omega_{2,2}\Omega_{1,1}: - 1/2 :\Omega_{3,5}\Omega_{2,2}\Omega_{1,1}: $$ $$+ 
 1/30 :\Omega_{1,7}\Omega_{0,3}\Omega_{0,3}: - 3/4 :\Omega_{2,6}\Omega_{0,3}\Omega_{0,3}: + 
 9/35 :\Omega_{1,7}\Omega_{1,2}\Omega_{0,3}: - 3/5 :\Omega_{2,6}\Omega_{1,2}\Omega_{0,3}: $$ $$+ 
 7/6 :\Omega_{3,5}\Omega_{1,2}\Omega_{0,3}: - 3/40 :\Omega_{0,8}\Omega_{1,2}\Omega_{1,2}:+ 
 1/2 :\Omega_{3,5}\Omega_{1,2}\Omega_{1,2}: - 5/6 :\Omega_{1,6}\Omega_{3,3}\Omega_{0,1}: $$ $$- 
 13/15 :\Omega_{2,5}\Omega_{3,3}\Omega_{0,1}: - 3/5 :\Omega_{1,6}\Omega_{2,3}\Omega_{0,2}: + 
 1/6 :\Omega_{3,4}\Omega_{2,3}\Omega_{0,2}: - 24/35 :\Omega_{0,7}\Omega_{2,3}\Omega_{1,1}: $$ $$+ 
 4/5 :\Omega_{2,5}\Omega_{2,3}\Omega_{1,1}: - 1/2 :\Omega_{3,4}\Omega_{2,3}\Omega_{1,1}: + 
 5/6 :\Omega_{1,6}\Omega_{1,3}\Omega_{0,3}: + 13/15 :\Omega_{2,5}\Omega_{1,3}\Omega_{0,3}: $$ $$+ 
 3/5 :\Omega_{1,6}\Omega_{2,2}\Omega_{0,3}: - 1/6 :\Omega_{3,4}\Omega_{2,2}\Omega_{0,3}: + 
 24/35 :\Omega_{0,7}\Omega_{1,3}\Omega_{1,2}: - 4/5 :\Omega_{2,5}\Omega_{1,3}\Omega_{1,2}: $$ $$+ 
 1/2 :\Omega_{3,4}\Omega_{1,3}\Omega_{1,2}: + 61/30 :\Omega_{1,5}\Omega_{3,3}\Omega_{0,2}: - 
 1/6 :\Omega_{2,4}\Omega_{3,3}\Omega_{0,2}: + 5/12 :\Omega_{3,3}\Omega_{0,6}\Omega_{1,1}: $$ $$- 
 1/4 :\Omega_{2,4}\Omega_{3,3}\Omega_{1,1}: - 61/30 :\Omega_{1,5}\Omega_{2,3}\Omega_{0,3}: + 
 1/6 :\Omega_{2,4}\Omega_{2,3}\Omega_{0,3}:- 1/10 :\Omega_{3,3}\Omega_{0,5}\Omega_{1,2}: $$ $$+ :\Omega_{3,3}\Omega_{1,4}\Omega_{1,2}: + 1/15 :\Omega_{0,6}\Omega_{2,3}\Omega_{1,2}: - 4/5 :\Omega_{1,5}\Omega_{2,3}\Omega_{1,2}:+ 1/12 :\Omega_{3,3}\Omega_{2,2}\Omega_{0,4}: $$ $$- 5/12 :\Omega_{0,6}\Omega_{1,3}\Omega_{1,3}:+ 1/4 :\Omega_{2,4}\Omega_{1,3}\Omega_{1,3}: - 1/15 :\Omega_{0,6}\Omega_{2,2}\Omega_{1,3}: + 4/5 :\Omega_{1,5}\Omega_{2,2}\Omega_{1,3}: $$ $$- 1/6 :\Omega_{3,3}\Omega_{2,2}\Omega_{1,3}: - 
 1/12 :\Omega_{2,3}\Omega_{2,3}\Omega_{0,4}: + 1/10 :\Omega_{2,3}\Omega_{0,5}\Omega_{1,3}:- 
 :\Omega_{1,4}\Omega_{2,3}\Omega_{1,3}: $$ $$+ 1/6 :\Omega_{2,3}\Omega_{2,3}\Omega_{1,3}:,$$

$$D^4_0 = \frac{451}{114660} :\Omega_{0,0}\Omega_{1,15}: + \frac{6251}{102960} :\Omega_{0,0}\Omega_{2,14}: + 
 \frac{10261}{69300} :\Omega_{0,0}\Omega_{3,13}: - \frac{1}{140} :\Omega_{0,0}\Omega_{6,10}: $$ $$+ 
 \frac{137}{840} :\Omega_{0,0}\Omega_{7,9}: + \frac{4849}{22050} :\Omega_{0,0}\Omega_{8,8}: - 
 \frac{451}{114660} :\Omega_{0,1}\Omega_{0,15}: + \frac{96}{2695} :\Omega_{0,1}\Omega_{1,14}: $$ $$- 
 \frac{1857}{28600} :\Omega_{0,1}\Omega_{2,13}: - \frac{467}{2475} :\Omega_{0,1}\Omega_{3,12}:- 
 \frac{1}{560} :\Omega_{0,1}\Omega_{5,10}: - \frac{347}{1680} :\Omega_{0,1}\Omega_{6,9}: $$ $$- 
 \frac{643}{9800} :\Omega_{0,1}\Omega_{7,8}: + \frac{719}{15015} :\Omega_{0,2}\Omega_{0,14}: + 
 \frac{24929}{1201200} :\Omega_{0,2}\Omega_{1,13}: + \frac{155}{1584} :\Omega_{0,2}\Omega_{2,12}:$$ $$ - \frac{1427}{9900} :\Omega_{0,2}\Omega_{3,11}:+ \frac{203}{720} :\Omega_{0,2}\Omega_{5,9}: + 
 \frac{209}{8400} :\Omega_{0,2}\Omega_{6,8}: - \frac{1}{10} :\Omega_{0,2}\Omega_{7,7}: $$ $$+ \frac{59}{1470} :\Omega_{1,1}\Omega_{0,14}: + \frac{811}{6300} :\Omega_{1,1}\Omega_{2,12}:+ \frac{191}{1260} :\Omega_{1,1}\Omega_{3,11}:+  \frac{1}{112} :\Omega_{1,1}\Omega_{4,10}: $$ $$- \frac{23}{105} :\Omega_{1,1}\Omega_{5,9}: + \frac{5}{48} :\Omega_{1,1}\Omega_{6,8}: +  \frac{1809}{9800} :\Omega_{1,1}\Omega_{7,7}: - \frac{4271}{69300} :\Omega_{0,3}\Omega_{0,13}: + $$ $$
 \frac{118}{4725} :\Omega_{0,3}\Omega_{1,12}: + \frac{22669}{118800} :\Omega_{0,3}\Omega_{2,11}: + 
  \frac{1279}{10800} :\Omega_{0,3}\Omega_{3,10}: +  \frac{257}{6300} :\Omega_{0,3}\Omega_{5,8}: $$ $$+ 
  \frac{1537}{4200} :\Omega_{0,3}\Omega_{6,7}: +  \frac{2467}{46200} :\Omega_{1,2}\Omega_{0,13}: + 
  \frac{2431}{25200} :\Omega_{1,2}\Omega_{1,12}:+  \frac{591}{2200} :\Omega_{1,2}\Omega_{2,11}: $$ $$ + 
  \frac{3523}{6300} :\Omega_{1,2}\Omega_{3,10}: +  \frac{5}{8} :\Omega_{1,2}\Omega_{4,9}: + 
  \frac{2383}{8400} :\Omega_{1,2}\Omega_{5,8}: +  \frac{3}{700} :\Omega_{1,2}\Omega_{6,7}:  $$ $$+  \frac{7}{96} :\Omega_{0,4}\Omega_{2,10}: - 
  \frac{367}{504} :\Omega_{0,4}\Omega_{3,9}: +  \frac{1019}{18900} :\Omega_{1,3}\Omega_{0,12}: - 
  \frac{11}{1260} :\Omega_{1,3}\Omega_{1,11}: $$ $$ -  \frac{4399}{3600} :\Omega_{1,3}\Omega_{2,10}: - \frac{5851}{7560} :\Omega_{1,3}\Omega_{3,9}: -  \frac{139}{210} :\Omega_{1,3}\Omega_{4,8}: - 
  \frac{2447}{4200} :\Omega_{1,3}\Omega_{5,7}: $$ $$+  \frac{1}{6} :\Omega_{1,3}\Omega_{6,6}: -  \frac{155}{2376} :\Omega_{2,2}\Omega_{0,12}: - 
  \frac{283}{1100} :\Omega_{2,2}\Omega_{1,11}: +  \frac{271}{1080} :\Omega_{2,2}\Omega_{3,9}: $$ $$+  \frac{1}{48} :\Omega_{2,2}\Omega_{4,8}: + 
  \frac{1}{10} :\Omega_{2,2}\Omega_{5,7}:+  \frac{7}{150} :\Omega_{2,2}\Omega_{6,6}: -  \frac{17}{560} :\Omega_{0,5}\Omega_{1,10}: $$ $$+ 
  \frac{83}{240} :\Omega_{0,5}\Omega_{2,9}: +  \frac{799}{2100} :\Omega_{0,5}\Omega_{3,8}: +  \frac{17}{56} :\Omega_{1,4}\Omega_{1,10}: + 
  \frac{17}{16} :\Omega_{1,4}\Omega_{2,9}: $$ $$+  \frac{27}{35} :\Omega_{1,4}\Omega_{3,8}:-  \frac{169}{5400} :\Omega_{2,3}\Omega_{0,11}: + 
  \frac{1021}{3150} :\Omega_{2,3}\Omega_{1,10}: -  \frac{163}{540} :\Omega_{2,3}\Omega_{2,9}: $$ $$- 
  \frac{2539}{2520} :\Omega_{2,3}\Omega_{3,8}:-  \frac{1}{4} :\Omega_{2,3}\Omega_{4,7}: -  \frac{59}{100} :\Omega_{2,3}\Omega_{5,6}: + 
  \frac{4691}{10080} :\Omega_{0,6}\Omega_{1,9}: $$ $$-  \frac{39899}{25200} :\Omega_{0,6}\Omega_{2,8}: - 
  \frac{1889}{5040} :\Omega_{0,6}\Omega_{3,7}: +  \frac{1037}{1680} :\Omega_{1,5}\Omega_{0,10}: - 
  \frac{22}{105} :\Omega_{1,5}\Omega_{1,9}:$$ $$-  \frac{1}{15} :\Omega_{1,5}\Omega_{2,8}: -  \frac{1369}{4200} :\Omega_{1,5}\Omega_{3,7}: - 
  \frac{17}{336} :\Omega_{2,4}\Omega_{0,10}: -  \frac{37}{96} :\Omega_{2,4}\Omega_{1,9}: $$ $$-  \frac{13}{240} :\Omega_{2,4}\Omega_{2,8}: + 
  \frac{419}{1680} :\Omega_{2,4}\Omega_{3,7}: +  \frac{7}{400} :\Omega_{3,3}\Omega_{0,10}: + 
  \frac{137}{1260} :\Omega_{3,3}\Omega_{1,9}: $$ $$+  \frac{3397}{5040} :\Omega_{3,3}\Omega_{2,8}: +  \frac{7}{90} :\Omega_{3,3}\Omega_{3,7}: + 
  \frac{19}{72} :\Omega_{3,3}\Omega_{4,6}: +  \frac{937}{1800} :\Omega_{3,3}\Omega_{5,5}:$$ $$+ 
  \frac{9697}{29400} :\Omega_{0,7}\Omega_{1,8}: +  \frac{37}{60} :\Omega_{0,7}\Omega_{2,7}: + 
  \frac{301}{450} :\Omega_{0,7}\Omega_{3,6}: -  \frac{967}{5040} :\Omega_{1,6}\Omega_{0,9}:  $$ $$-  \frac{47}{240} :\Omega_{1,6}\Omega_{1,8}: + 
  \frac{123}{700} :\Omega_{1,6}\Omega_{2,7}: +  \frac{4}{5} :\Omega_{1,6}\Omega_{3,6}: -  \frac{13}{360} :\Omega_{2,5}\Omega_{0,9}:  $$ $$- 
  \frac{4733}{8400} :\Omega_{2,5}\Omega_{1,8}: -  \frac{4}{5} :\Omega_{2,5}\Omega_{2,7}: +  \frac{17}{225} :\Omega_{2,5}\Omega_{3,6}: + 
  \frac{11}{252} :\Omega_{3,4}\Omega_{0,9}:$$ $$+  \frac{27}{140} :\Omega_{3,4}\Omega_{1,8}: +  \frac{163}{420} :\Omega_{3,4}\Omega_{2,7}: + 
  \frac{59}{180} :\Omega_{3,4}\Omega_{3,6}: -  \frac{1199}{2352} :\Omega_{0,8}\Omega_{1,7}: $$ $$+ 
  \frac{11969}{16800} :\Omega_{0,8}\Omega_{2,6}: -  \frac{4399}{6300} :\Omega_{0,8}\Omega_{3,5}: - 
  \frac{45}{196} :\Omega_{1,7}\Omega_{1,7}: -  \frac{97}{175} :\Omega_{1,7}\Omega_{2,6}: $$ $$-  \frac{857}{4200} :\Omega_{1,7}\Omega_{3,5}: + 
  \frac{1}{6} :\Omega_{2,6}\Omega_{2,6}: -  \frac{103}{360} :\Omega_{2,6}\Omega_{3,5}: -  \frac{127}{300} :\Omega_{3,5}\Omega_{3,5}:,$$

$$D^2_0 = -\frac{2419}{705600} \Omega_{0,18} +  \frac{2356854113}{13722508800} \partial^2 \Omega_{0,16} - 
 \frac{3876250811}{34306272000}  \partial^2 \Omega_{1,15} + \frac{710040893}{4678128000} \partial^2 \Omega_{2,14} $$ $$ - 
 \frac{3598850419}{44108064000} \partial^2 \Omega_{3,13} + \frac{50867963}{212058000} \partial^2 \Omega_{4,12} + 
 \frac{5617559}{75398400} \partial^2 \Omega_{5,11} - \frac{47629609}{62832000}\partial^2 \Omega_{6,10} $$ $$+ 
 \frac{12154709537}{7916832000} \partial^2 \Omega_{7,9} - \frac{13317559687}{15833664000} \partial^2 \Omega_{8,8}.$$
 
A computer calculation shows that $D^8_0 + D^6_0 + D^4_0 + D^2_0$ lies in $Ker(\pi_3)$, so it must coincide with $D_0$. In particular, $R_0 = -\frac{2419}{705600} J^{18}$.

\end{document}